\documentclass[12pt]{amsart}

  \usepackage{graphicx}
  \usepackage{pictexwd,dcpic}
  \usepackage{latexsym,amssymb,epsfig,a4wide}

  \def\kk{{\mathbb k}}
  
  \def\NN{{\mathbb N}}
  
  \def\ZZ{{\mathbb Z}}
  \def\RR{{\mathbb R}}
  \def\kk{{\mathbf k}}
  \def\aA{{\mathcal A}}
  \def\cC{{\mathcal C}}
  \def\dD{{\mathcal D}}
  
  \def\fF{{\mathcal F}}

  \def\des{{\rm des}}
  \def\Des{{\rm Des}}
  \def\asc{{\rm asc}}
  \def\exc{{\rm exc}}
  \def\iexc{{\rm iexc}}
  \def\inte{{\rm int}}
  
  \def\link{{\rm link}}

  \def\sd{{\rm sd}}
  
  \def\sm{\smallsetminus}

  \theoremstyle{plain}
    \newtheorem{theorem}{Theorem}[section]
    \newtheorem{proposition}[theorem]{Proposition}
    \newtheorem{lemma}[theorem]{Lemma}
    \newtheorem{corollary}[theorem]{Corollary}
    \newtheorem{conjecture}[theorem]{Conjecture}
  \theoremstyle{definition}
    \newtheorem{definition}[theorem]{Definition}
    \newtheorem{example}[theorem]{Example}

    \newtheorem{remark}[theorem]{Remark}

  \numberwithin{equation}{section}

\begin{document}

\title[A decomposition of the derangement polynomial of type $B$]{A symmetric
unimodal decomposition of the derangement polynomial of type $B$}

  \author{Christos~A.~Athanasiadis}
  \author{Christina~Savvidou}

  \address{Department of Mathematics
          (Division of Algebra-Geometry) \\
          University of Athens \\
          Panepistimioupolis, Athens 15784 \\
          Hellas (Greece)}
  \email{caath@math.uoa.gr, savvtina@math.uoa.gr}

\date{March 15, 2013}
\thanks{2000 \textit{Mathematics Subject Classification.} Primary 05A05; \, Secondary
05A15, 05E45.}
\keywords{Signed permutation, Eulerian polynomial, derangement polynomial, excedance,
barycentric subdivision, local $h$-vector, $\gamma$-vector}

  \begin{abstract}
    The derangement polynomial $d_n (x)$ for the symmetric group enumerates
    derangements by the number of excedances. The derangement polynomial $d^B_n
    (x)$ for the hyperoctahedral group is a natural type $B$ analogue. A new
    combinatorial formula for this polynomial is given in this paper. This formula
    implies that $d^B_n (x)$ decomposes as a sum of two nonnegative, symmetric and
    unimodal polynomials whose centers of symmetry differ by a half and thus provides
    a new transparent proof of its unimodality. A geometric interpretation, analogous
    to Stanley's interpretation of $d_n (x)$ as the local $h$-polynomial of the
    barycentric subdivision of the simplex, is given to one of the summands of
    this decomposition. This interpretation leads to a unimodal decomposition of the 
    Eulerian polynomial of type $B$ whose summands can be expressed in terms of the 
    Eulerian polynomial of type $A$. The various decomposing polynomials introduced 
    here are also studied in terms of recurrences, generating functions, 
    combinatorial interpretations, expansions and real-rootedness.
  \end{abstract}

  \maketitle

\section{Introduction and results}
\label{sec:intro}

The derangement polynomial of order $n$ is an interesting $q$-analogue of the
number of derangements (elements without fixed points) in the symmetric group
$\mathfrak{S}_n$. It is defined by the formula
  \begin{equation} \label{eq:dndef}
    d_n (x) \ = \ \sum_{w \in \dD_n} \ x^{\exc(w)},
  \end{equation}
where $\exc(w)$ is the number of excedances (see Section \ref{sec:perms} for
missing definitions) of $w \in \mathfrak{S}_n$ and $\dD_n$ is the set of
derangements in $\mathfrak{S}_n$. The polynomial $d_n(x)$, first studied by
Brenti \cite{Bre90} in the context of symmetric functions, has a number of
pleasant properties. For instance, it has symmetric and unimodal coefficients
\cite{Bre90} (see also \cite[Section~4]{AS12} \cite[Section~5]{SW10}
\cite{Ste92}) and only real roots \cite{Zh95}. It can also be expressed as
  \begin{equation} \label{eq:dnA}
    d_n(x) \ = \ \sum_{k=0}^n \, (-1)^{n-k} \binom{n}{k} A_k(x),
  \end{equation}
where $A_k(x) = \sum_{w \in \mathfrak{S}_k} x^{\des(w)}$ is the $k$-th Eulerian
polynomial.

We will be concerned with a natural analogue of $d_n(x)$ for the hyperoctahedral
group $B_n$ of signed permutations, introduced and studied independently by Chen,
Tang and Zhao \cite{CTZ09} and by Chow \cite{Ch09}. It is defined by the formula
  \begin{equation} \label{eq:dnBdef}
    d^B_n (x) \ = \ \sum_{w \in \dD^B_n} \ x^{\exc_B(w)},
  \end{equation}
where $\exc_B(w)$ is the number of type $B$ excedances of $w \in B_n$, introduced
by Brenti \cite{Bre94}, and $\dD^B_n$ is the set of derangements in $B_n$.

The derangement polynomial $d^B_n(x)$ shares most of the main properties of $d_n
(x)$. For instance, it is real-rooted \cite{CTZ09, Ch09}, hence it has unimodal
(but not symmetric) coefficients, and satisfies the analogue
  \begin{equation} \label{eq:dnB}
    d^B_n(x) \ = \ \sum_{k=0}^n \, (-1)^{n-k} \binom{n}{k} B_k(x)
  \end{equation}
of (\ref{eq:dnA}), where $B_k(x) = \sum_{w \in B_k} x^{\des_B(w)}$ is the $k$-th
Eulerian polynomial of type $B$. Our first main result is the following
combinatorial formula for $d^B_n(x)$.

\begin{theorem} \label{thm:main}
  We have
    \begin{equation} \label{eq:main}
      d_n^B(x) \ = \ \sum \ {n \choose r_0, r_1,\dots,r_k} \, x^{\lfloor
      \frac{k+1}{2} \rfloor} \, d_{r_0}(x) \, A_{r_1}(x) \cdots A_{r_k}(x)
    \end{equation}
  for $n \in \NN$, where $A_0 (x) = 0$, $d_0 (x) = 1$ and the sum ranges over all
  $k \in \NN$ and over all sequences $(r_0, r_1,\dots,r_k)$ of nonnegative integers
  which sum to $n$.
\end{theorem}

Chow \cite[Section 4]{Ch09} gave an additional proof of the unimodality of $d^B_n
(x)$ by expressing it as a sum of certain nonnegative unimodal polynomials, defined
by a symmetric function identity, of a common mode. Theorem \ref{thm:main} implies
that $d^B_n (x)$ can be written as a sum of two polynomials with nonnegative,
symmetric and unimodal coefficients, whose centers of symmetry differ by a half, and
thus provides a new proof of its unimodality, as we now explain. Since $d^B_n (x)$
has degree $n$ and zero constant term, it can be written uniquely in the form
  \begin{equation} \label{eq:dnBsymsum}
    d^B_n(x) \ = \ f^+_n (x) \, + \, f^-_n (x),
  \end{equation}
where $f^+_n (x)$ and $f^-_n (x)$ are polynomials of degrees at most $n-1$ and $n$,
respectively, satisfying
  \begin{eqnarray}
  \label{eq:symf+}  f^+_n (x) & = & x^n f^+_n (1/x) \\
  \label{eq:symf-}  f^-_n (x) & = & x^{n+1} f^-_n (1/x)
\end{eqnarray}
(see, for instance, \cite[Lemma 2.4]{BSt10} for this elementary fact). For the
first few values of $n$ we have
  $$ f^+_n (x) \ = \ \begin{cases}
    1, \ \ & \text{if \ $n=0$} \\
    0, \ \ & \text{if \ $n=1$} \\
    3x, \ \ & \text{if \ $n=2$} \\
    7x + 7x^2, \ \ & \text{if \ $n=3$} \\
    15x + 87x^2 + 15x^3, \ \ & \text{if \ $n=4$} \\
    31x + 551x^2+ 551x^3 + 31x^4, \ \ & \text{if \ $n=5$} \\
    63x + 2803x^2 + 8243x^3 + 2803x^4 + 63x^5, \ \ & \text{if \ $n=6$} \\
    127x + 12867x^2 + 84827x^3 + 84827x^4 + 12867x^5 + 127x^6, \ \ & \text{if \
    $n=7$} \end{cases} $$
and
  $$ f^-_n (x) \ = \ \begin{cases}
    0, \ \ & \text{if \ $n=0$} \\
    x, \ \ & \text{if \ $n=1$} \\
    x + x^2, \ \ & \text{if \ $n=2$} \\
    x + 13x^2 + x^3, \ \ & \text{if \ $n=3$} \\
    x + 57x^2 + 57x^3 + x^4, \ \ & \text{if \ $n=4$} \\
    x + 201x^2 + 761x^3 + 201x^4 + x^5, \ \ & \text{if \ $n=5$} \\
    x + 653x^2 + 6333x^3 + 6333x^4 + 653x^5 + x^6, \ \ & \text{if \ $n=6$} \\
    x + 2045x^2 + 42757x^3 + 106037x^4 + 42757x^5 + 2045x^6 + x^7, \ \ & \text{if \
    $n=7$.} \end{cases} $$
The following information for the polynomials $f^+_n (x)$ and $f^-_n (x)$ and for
$d^B_n (x)$ can be derived from (\ref{eq:main}).

\begin{corollary} \label{cor:main}
  We have
    \begin{equation} \label{eq:main+}
      f^+_n (x) \ = \ \sum \ {n \choose r_0, r_1,\dots,r_{2k}} \, x^k \, d_{r_0}
     (x) \, A_{r_1} (x) \cdots A_{r_{2k}}(x)
    \end{equation}
  and
    \begin{equation} \label{eq:main-}
      f^-_n (x) \ = \ \sum \ {n \choose r_0, r_1,\dots,r_{2k+1}} \, x^{k+1} \,
      d_{r_0}(x) \, A_{r_1} (x) \cdots A_{r_{2k+1}}(x)
    \end{equation}
  for $n \in \NN$, where the sums range over all $k \in \NN$ and over all sequences
  $(r_0, r_1,\dots,r_{2k})$ {\rm (}respectively, $(r_0, r_1,\dots,r_{2k+1})${\rm )}
  of nonnegative integers which sum to $n$.
  Moreover, $f^+_n (x)$ and $f^-_n (x)$ are $\gamma$-nonnegative, meaning there
  exist nonnegative integers $\xi^+_{n,i}$ and $\xi^-_{n,i}$ such that

    \begin{equation} \label{eq:main++}
      f^+_n (x) \ = \ \sum_{i=0}^{\lfloor n/2 \rfloor} \ \xi^+_{n,i} \, x^i
      (1+x)^{n-2i}
    \end{equation}
  and
    \begin{equation} \label{eq:main--}
      f^-_n (x) \ = \ \sum_{i=0}^{\lfloor (n+1)/2 \rfloor} \ \xi^-_{n,i} \, x^i
      (1+x)^{n+1-2i}.
    \end{equation}

  In particular, $f^+_n (x)$ and $f^-_n (x)$ are symmetric and unimodal, with
  center of symmetry $n/2$ and $(n+1)/2$, respectively, and $d_n^B(x)$ is
  unimodal with a peak at $\lfloor (n+1)/2 \rfloor$.
\end{corollary}

Much of the motivation behind this paper comes from the theory of subdivisions
and local $h$-vectors, developed by Stanley \cite{Sta92}, and its extension
\cite{Ath12}. We recall that the local $h$-vector is a fundamental enumerative
invariant of a simplicial subdivision (triangulation) of the simplex. An example
by Stanley (see \cite[Proposition~2.4]{Sta92}) shows that $d_n(x)$ is equal to
the local $h$-polynomial of the (first) simplicial barycentric subdivision of the
$(n-1)$-dimensional simplex. This fact gives a geometric interpretation to $d_n
(x)$ and another proof of its symmetry and unimodality.

  \begin{figure}
  \epsfysize = 3.0 in \centerline{\epsffile{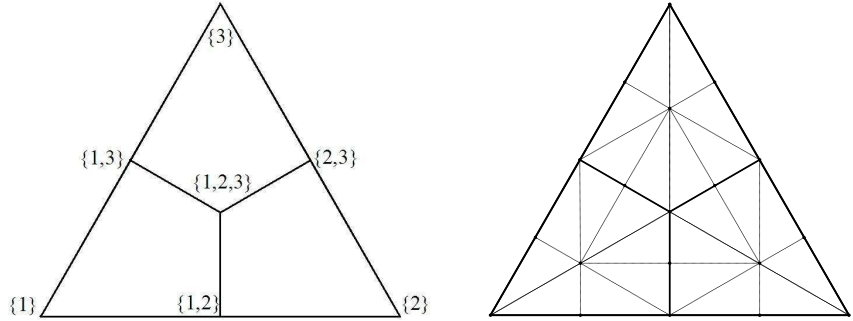}}
  \caption{The cubical barycentric subdivision of the 2-simplex and its
  barycentric subdivision $K_3$}
  \label{fig:K3}
  \end{figure}

Our second main result provides a type $B$ analogue to this interpretation. To
state it, we introduce the following notation. We denote by $K_n$ the simplicial
barycentric subdivision of the cubical barycentric subdivision of the
$(n-1)$-dimensional simplex (Figure \ref{fig:K3} shows this subdivision for
$n=3$). We also introduce the `half Eulerian polynomials'
  \begin{equation} \label{eq:Bn+def}
    B^+_n (x) \ = \ \sum_{w \in B^+_n} \ x^{\des_B(w)}
  \end{equation}
and
  \begin{equation} \label{eq:Bn-def}
    B^-_n (x) \ = \ \sum_{w \in B^-_n} \ x^{\des_B(w)}
  \end{equation}
for the group $B_n$, where $B^+_n$ and $B^-_n$ are the sets of signed permutations
of length $n$ with positive and negative, respectively, last entry, and set $B^+_0
(x) = 1$ and $B^-_0 (x) = 0$ (the set $B^+_n$ has appeared in the context of major
indices for classical Weyl groups; see \cite[page~613]{BC04}).

\begin{theorem} \label{thm:localint}
  The polynomial $f^+_n (x)$ is equal to the local $h$-polynomial of the simplicial
  subdivision $K_n$ (in particular, $f^+_n (x)$ has nonnegative, symmetric and
  unimodal coefficients). Moreover, we have
    \begin{equation} \label{eq:localint+}
      f^+_n (x) \ = \ \sum_{k=0}^n \, (-1)^{n-k} {n \choose k} \, B^+_k (x)
    \end{equation}
  and
    \begin{equation} \label{eq:localint-}
      f^-_n (x) \ = \ \sum_{k=0}^n \, (-1)^{n-k} {n \choose k} \, B^-_k (x)
    \end{equation}
  for $n \in \NN$.
\end{theorem}

We should point out that it is Theorem \ref{thm:localint} and the methods of
\cite{Ath12, Sta92} which led the authors to suspect that formula (\ref{eq:main})
holds. Indeed, it follows from the relevant definitions and some more work (see
Section \ref{sec:geom}) that the local $h$-polynomial of $K_n$ is equal to the
right-hand side of (\ref{eq:localint+}). By exploiting the symmetry of this
polynomial and certain recurrence relations for that and for $d^B_n (x)$ (see
Section~\ref{sec:half}), one can show that the local $h$-polynomial of $K_n$ is
equal to $f^+_n (x)$, as defined by the decomposition (\ref{eq:dnBsymsum}). A
formula for the change in the local $h$-vector of a simplicial subdivision of
the simplex after further subdivision \cite[Proposition 3.6]{Ath12} (see also
Proposition~\ref{prop:localrelformula}) can then be used to produce equation
(\ref{eq:main+}). This suggested that (\ref{eq:main-}), and hence
(\ref{eq:main}), hold as well.

The structure and other results of this paper are as follows.
Section~\ref{sec:perms} provides the necessary background on (signed)
permutations, simplicial complexes and subdivisions. Section~\ref{sec:proof}
proves Theorem~\ref{thm:main} and Corollary~\ref{cor:main}. A bijective proof
of Theorem~\ref{thm:main}, as well as one using generating functions, is
given and the exponential generating functions of $f^+_n (x)$ and $f^-_n
(x)$ are computed. Section \ref{sec:comb} gives a combinatorial interpretation
to the coefficients of these polynomials. Section \ref{sec:relative} proves
the main properties of the relative local $h$-vector, a generalization of
the concept of local $h$-vector which was introduced in \cite[Section
3]{Ath12} (and, in a variant form, in \cite{Ni12}) and derives a monotonicity
property for local $h$-vectors. These results were stated without proof in
\cite{Ath12}. As an example (used in one of the proofs of
Theorem~\ref{thm:localint}), the relative local $h$-vector of the barycentric
subdivision of the simplex is computed. Section \ref{sec:geom} gives two proofs
of Theorem~\ref{thm:localint}. A first step towards these proofs is to
interpret $B^+_n (x)$ as the $h$-polynomial of the simplicial complex $K_n$
(Proposition~\ref{prop:Knhpoly}). Given that, one proof uses the theory of
(relative) local $h$-vectors, as discussed earlier, while the other uses
recurrences and generating functions.

Section~\ref{sec:half} studies the polynomials $B^+_n (x)$ and $B^-_n (x)$.
A simple relation between the two is shown to hold (Lemma~\ref{lem:rec}).
Using its interpretation as the $h$-polynomial of $K_n$ and the theory of
local $h$-vectors, a simple formula for $B^+_n (x)$ (hence one for $B^-_n
(x)$ and one for the Eulerian polynomial $B_n (x)$) in terms of the
Eulerian polynomial $A_n (x)$ is proven (Proposition~\ref{prop:Bn+formula}).
Using this formula, it is shown that $B^+_n (x)$ and $B^-_n (x)$ are
real-rooted, hence unimodal and log-concave, and a new proof of the
unimodality of $B_n (x)$ is deduced. Recurrences and generating functions
for $B^+_n (x)$ and $B^-_n (x)$, as well as for $f^+_n (x)$ and $f^-_n (x)$,
are also given and a third proof of Theorem~\ref{thm:localint} is deduced.

\section{Permutations and subdivisions}
\label{sec:perms}

This section fixes notation and includes background material on (signed)
permutations, simplicial complexes and their subdivisions. For more
information on these topics, the reader is referred to \cite{Bj95, BB05,
Sta92, StaCCA, StaEC1}.

Throughout this paper, $\mathbb{N}$ denotes the set of nonnegative integers.
For each positive integer $n$ we set $[n]:= \{1, 2,\dots,n\}$ and $\Omega_n =
\{1, -1, 2, -2,\dots,n, -n\}$. We denote by $|S|$ the cardinality, and by
$2^S$ the set of all subsets, of a finite set $S$.

\subsection{Permutations}
\label{subsec:perm}

A \emph{permutation} of a finite set $S$ is a bijective map $w: S \to S$.
We denote by $\mathfrak{S} (S)$ the set of all permutations of $S$ and set
$\mathfrak{S}_n := \mathfrak{S} ([n])$.
Suppose that $S = \{a_1, a_2,\dots,a_n\}$ has $n$ elements, which are totally
ordered by $a_1 \prec a_2 \prec \cdots \prec a_n$. A permutation $w \in
\mathfrak{S}(S)$ can be represented as the sequence $(w(a_1),
w(a_2),\dots,w(a_n))$, or as the word $w(a_1) w(a_2) \cdots w(a_n)$, or as a
disjoint union of cycles \cite[Section~1.3]{StaEC1}. The \emph{standard cycle
form} is defined by requiring that (a) each cycle is written with its largest
element (with respect to the total order $\preceq$) first and (b) the cycles
are written in increasing order of their largest element \cite[page~23]{StaEC1}.

Given $w \in \mathfrak{S} (S)$, an element $a \in S$ is called an \emph{excedance}
of $w$ (with respect to $\preceq$) if $w(a) \succ a$ and an \emph{inverse
excedance} if $w(a) \prec a$. The element $a_i \in S$ is called a \emph{descent}
(respectively, \emph{ascent}) of $w$ if $i \in [n-1]$ and $w(a_i) \succ w(a_{i+1})$
(respectively, $w(a_i) \prec w(a_{i+1})$). The number of excedances (respectively,
inverse excedances, descents or ascents) of $w$ will be denoted by $\exc (w)$
(respectively, by $\iexc(w)$, $\des (w)$ or $\asc (w)$). The $n$th Eulerian
polynomial \cite[Section~1.4]{StaEC1} is defined by the formulas
  \begin{equation} \label{eq:eulerdef}
    A_n (x) \ = \ \sum_{w \in \mathfrak{S} (S)} x^{\exc(w)} \ = \
                  \sum_{w \in \mathfrak{S} (S)} x^{\iexc(w)} \ = \
                  \sum_{w \in \mathfrak{S} (S)} x^{\des(w)} \ = \
                  \sum_{w \in \mathfrak{S} (S)} x^{\asc(w)}.
  \end{equation}
Clearly, these sums depend only on $n$ and not on $S$ or the choice of total
order $\preceq$.

The previous definitions apply in particular to $\mathfrak{S}_n$ (with the
standard choice of $\preceq$ obtained by setting $a_i = i$ for $1 \le i \le
n$). We will denote by $\dD_n$ the set of all derangements (permutations
without fixed points) in $\mathfrak{S}_n$.

\subsection{Signed permutations}
\label{subsec:signed}

For the purposes of this paper, it will be convenient to define a \emph{signed
permutation} of $[n]$ as a choice of a subset $S = \{a_1, a_2,\dots,a_n\}$ of
$\Omega_n$ such that $a_i \in \{i, -i\}$ for $1 \le i \le n$ and permutation $w
\in \mathfrak{S} (S)$. We will represent such a permutation $w$ as the sequence
$(w(a_1), w(a_2),\dots,w(a_n))$, or as the word $w(a_1) w(a_2) \cdots w(a_n)$,
or as a disjoint union of cycles. We will find it convenient to define the
standard cycle form of $w$ using the total order on $S$ which is the reverse of
the one inherited from the natural total order on $\ZZ$. Thus, cycles of $w$ will
be written with their \emph{smallest} element first and in \emph{decreasing} order
of their smallest element. We will say that $w$ is a \emph{derangement} if there
is no $a \in S \cap [n]$ such that $w(a) = a$. We will denote the set of all
signed permutations of $[n]$ by $B_n$ and the set of all derangements in $B_n$
by $\dD^B_n$.

Given $w \in B_n$ as before, we say that $i \in \{0, 1,\dots,n-1\}$ is a
\emph{$B$-descent} (respectively, $B$-\emph{ascent}) of $w$ if $w(a_i) >
w(a_{i+1})$ (respectively, $w(a_i) < w(a_{i+1})$), where $w(a_0) = 0$ by
convention. The $n$th Eulerian polynomial of type $B$ \cite[Section~3]{Bre94}
can be defined by
  \begin{equation} \label{eq:eulerBdef}
    B_n (x) \ = \ \sum_{w \in B_n} x^{\des_B(w)}
            \ = \ \sum_{w \in B_n} x^{\asc_B(w)},
  \end{equation}
where $\des_B (w)$ stands for the number of $B$-descents and $\asc_B (w)$ for
the number of $B$-ascents of $w \in B_n$. Following Brenti~\cite[p.~431]{Bre94},
we say that $a \in S$ is a \emph{$B$-excedance} of $w$ if $w(a) > a$, or if $-a
\in [n]$ and $w(a) = a$. We say that $a \in S$ is an \emph{inverse $B$-excedance}
of $w$ if $w(a) < a$, or if $-a \in [n]$ and $w(a) = a$. The number of
$B$-excedances of $w$ will be denoted by $\exc_B(w)$ and that of inverse
$B$-excedances by $\iexc_B(w)$. We then have $\iexc_B(w) = \exc_B(w^{-1})$ and
(see Theorem~3.15 and Corollary~3.16 in \cite{Bre94})
  \begin{equation} \label{eq:eulerBdef2}
    B_n (x) \ = \ \sum_{w \in B_n} x^{\exc_B(w)}.
  \end{equation}

The $n$th derangement polynomial of type $B$ is defined by (\ref{eq:dnBdef}).
Since $\exc_B(w) = \iexc_B(w^{-1})$ and the map which sends a permutation $w \in
\mathfrak{S} (S)$ to its inverse $w^{-1}$ induces an involution on $B_n$ which
preserves fixed points, we have
  \begin{equation} \label{eq:dnBidef}
    d^B_n (x) \ = \ \sum_{w \in \dD^B_n} \ x^{\iexc_B(w)}.
  \end{equation}
For the similar reasons, (\ref{eq:dndef}) continues to hold if $\exc$ is
replaced by $\iexc$ and (\ref{eq:eulerBdef2}) continues to hold if $\exc_B$ is
replaced by $\iexc_B$.

\subsection{Polynomials}
\label{subsec:polys}

Let $p(x) = \sum_{k \ge 0} a_k x^k = \sum_{k=0}^d a_k x^k$ be a polynomial with real
coefficients. We recall that $p(x)$ is unimodal (and has unimodal coefficients)
if there exists an index $0 \le j \le d$ such that $a_i \le a_{i+1}$ for $0 \le
i \le j-1$ and $a_i \ge a_{i+1}$ for $j \le i \le d-1$. Such an index is called
a \emph{peak}. The polynomial $p(x)$ is said to be log-concave if $a_i^2 \ge
a_{i-1}a_{i+1}$ for $1 \le i \le d-1$ and to have internal zeros if there exist
indices $0 \le i < j < k \le d$ such that $a_i, a_k \ne 0$ and $a_j = 0$. We will
say that $p(x)$ is symmetric (and that it has symmetric coefficients) if there
exists an integer $n \ge d$ such that $a_i = a_{n-i}$ for $0 \le i \le n$. The
\emph{center of symmetry} of $p(x)$ is then defined to be $n/2$ (this is
well-defined provided $p(x)$ is nonzero).

We will say that $p(x)$ is \emph{real-rooted} if all its complex roots are real.
It is well-known (see, for instance, \cite{Sta89}) that if $p(x)$ is a
real-rooted polynomial with nonnegative coefficients, then $p(x)$ is log-concave
and unimodal, with no internal zeros. The following theorem, first proved by
Edrei~\cite{Ed53}, gives a necessary and sufficient condition for a polynomial
with nonnegative real coefficients to be real-rooted.
  \begin{theorem} {\rm (\cite{Ed53})} \label{thm:realroots}
    Let $p(x) = \sum_{k \ge 0} a_k x^k \in \RR[x]$ be a polynomial with $a_k \ge
    0$ for every $k \in \NN$ and set $a_k = 0$ for all negative integers $k$.
    Then $p(x)$ is real-rooted if and only if every minor of the lower triangular
    matrix $(a_{i-j})_{i,j=0}^\infty$ is nonnegative.
  \end{theorem}
A (nonzero) symmetric polynomial $p(x) \in \RR[x]$ can be written (uniquely) in
the form
  \begin{equation} \label{eq:defgamma}
      p(x) \ = \ (1+x)^n \ \gamma \left( \frac{x}{(1+x)^2} \right) \, = \,
      \sum_{i=0}^{\lfloor n/2 \rfloor} \, \gamma_i x^i (1+x)^{n-2i}
  \end{equation}
for some polynomial $\gamma(x) = \sum_{i \ge 0} \gamma_i x^i$. We say that $p(x)$
is \emph{$\gamma$-nonnegative} if $\gamma_i \ge 0$ for every $i$. Clearly, every
$\gamma$-nonnegative polynomial is unimodal. For classes of $\gamma$-nonnegative
polynomials which appear in combinatorics we refer the reader, for instance, to
\cite{DPK09} and references therein.

\subsection{Simplicial complexes}
\label{subsec:simpcomp}

An (\emph{abstract}) \emph{simplicial complex} $\Delta$ on the ground set $V$
is a collection $\Delta$ of subsets of $V$ such that $F \subseteq G \in \Delta$
implies $F \in \Delta$ (all simplicial complexes considered in this paper will
be assumed to be finite). The elements of $\Delta$ are called \emph{faces}. The
\emph{dimension} of a face is equal to one less than its cardinality. The
\emph{dimension} of $\Delta$ is the maximum dimension of its faces. Faces of
dimension 0 and 1 are called \emph{vertices} and \emph{edges}, respectively. A
\emph{facet} of $\Delta$ is a face which is maximal with respect to inclusion.
The complex $\Delta$ is said to be \emph{pure} if all its facets have the same
dimension. The \emph{face poset} $\mathcal{F}(\Delta)$ of a simplicial complex
$\Delta$ is the set of nonempty faces of $\Delta$, partially ordered by inclusion.

The \emph{open star} ${\rm st}_{\Delta}(F)$ of a face $F \in \Delta$ is the
collection of all faces of $\Delta$ containing $F$. The \emph{link} of a face
$F$ in $\Delta$ is the subcomplex of $\Delta$ defined as $\link_{\Delta}(F) = \{
G \sm F: G \in \Delta, F \subseteq G\}$. Suppose that $\Delta_1$ and $\Delta_2$
are simplicial complexes on disjoint ground sets. The \emph{simplicial join}
of $\Delta_1$ and $\Delta_2$ is the simplicial complex $\Delta_1 * \Delta_2$
whose faces are the sets of the form $F_1 \cup F_2$, where $F_1 \in \Delta_1$
and $F_2 \in \Delta_2$. The \emph{order complex} \cite[Section~9.3]{Bj95}
\cite[Section~3.8]{StaEC1} of a (finite) partially ordered set $Q$ is defined
as the simplicial complex of chains (totally ordered subsets) of $Q$.

All topological properties or invariants of $\Delta$ mentioned in the sequel
will refer to those of its geometric realization $\|\Delta\|$
\cite[Section~9.1]{Bj95}. For example, $\Delta$ is a \emph{simplicial ball} if
$\|\Delta\|$ is homeomorphic to a ball. For a simplicial $d$-dimensional ball
$\Delta$, we denote by $\partial \Delta$ the subcomplex consisting of all
subsets of the $(d-1)$-dimensional faces which are contained in a unique facet
of $\Delta$. We call $\partial \Delta$ the \emph{boundary} and $\inte(\Delta)
:= \Delta \sm \partial \Delta$ the \emph{interior} of $\Delta$.

\subsection{Subdivisions}
\label{subsec:sub}

Let $\Delta$ be a simplicial complex. A (\emph{topological}) \emph{simplicial
subdivision} of $\Delta$ \cite[Section~2]{Sta92} is a simplicial complex $\Delta'$
together with a map $\sigma: \Delta' \to \Delta$ such that the following hold for
every $F \in \Delta$: (a) the set $\Delta'_F := \sigma^{-1}
(2^F)$ is a subcomplex of $\Delta'$ which is a simplicial ball of dimension
$\dim(F)$; and (b) the interior of $\Delta'_F$ is equal to $\sigma^{-1}(F)$. The
subcomplex $\Delta'_F$ is called the \emph{restriction} of $\Delta'$ to $F$. The
face $\sigma(G)\in \Delta$ is called the \emph{carrier} of $G \in \Delta'$. The
subdivision $\Delta'$ is called \emph{quasi-geometric} \cite[Definition 4.1
(a)]{Sta92} if no face of $\Delta'$ has the carriers of its vertices contained in
a face of $\Delta$ of smaller dimension. Moreover, $\Delta'$ is called
\emph{geometric} \cite[Definition 4.1 (b)]{Sta92} if there exists a geometric
realization of $\Delta'$ which geometrically subdivides a geometric realization
of $\Delta$, in the way prescribed by $\sigma$. Clearly, all geometric
subdivisions (such as the barycentric subdivisions considered in this paper) are
quasi-geometric.

We now describe two common ways to subdivide a simplicial complex $\Delta$. The
order complex of the face poset $\fF(\Delta)$, denoted by $\sd (\Delta)$, consists
of the chains of nonempty faces of $\Delta$. This complex is naturally a (geometric)
simplicial subdivision of $\Delta$, called the \emph{barycentric subdivision}, where
the carrier of a chain $\cC$ of nonempty faces of $\Delta$ is defined as the maximum
element of $\cC$.

Given a face $F \in \Delta$ we set $\Delta' = (\Delta \sm \textrm{st}_{\Delta}(F))
\cup (\{v\} * \partial(2^F) * \textrm{lk}_{\Delta}(F))$, where $v$ is a new vertex
added and $\partial(2^F) = 2^F \sm F$. Then $\Delta'$ is a simplicial complex which
is a simplicial subdivision of $\Delta$, called the \emph{stellar subdivision} of
$\Delta$ on $F$.

\subsection{Face enumeration}
\label{subsec:enumface}

Let $\Delta$ be a $(d-1)$-dimensional simplicial complex. We denote by $f_i
(\Delta)$ the number of $i$-dimensional faces of $\Delta$. A fundamental
enumerative invariant of $\Delta$ is the \emph{$f$-polynomial}, defined by
  $$ f_{\Delta}(x) \ = \ \sum_{i=0}^{d-1} f_i (\Delta) x^i. $$
The \emph{$h$-polynomial} of $\Delta$ is defined by
  $$ h_{\Delta}(x) \ = \ \sum_{i=0}^d h_i(\Delta) \, x^i \ = \ (1-x)^d
     f_{\Delta} \left( \frac{x}{1-x} \right). $$
For the importance of $h$-polynomials, the reader is referred to
\cite[Chapter~II]{StaCCA}. For the simplicial join $\Delta_1 * \Delta_2$ of two
simplicial complexes we have $h(\Delta_1 * \Delta_2, x) = h(\Delta_1, x) h(\Delta_2,
x)$.

Let $\Gamma$ be a simplicial subdivision of a $(d-1)$-dimensional simplex $2^V$.
The polynomial $\ell_V (\Gamma, x) = \ell_0 + \ell_1 x + \cdots + \ell_d x^d$
defined by
  \begin{equation} \label{eq:deflocalh}
    \ell_V (\Gamma, x) \ = \sum_{F \subseteq V} \ (-1)^{d - |F|} \,
    h (\Gamma_F, x)
  \end{equation}
is the \emph{local $h$-polynomial} of $\Gamma$ (with respect to $V$)
\cite[Definition~2.1]{Sta92}. The sequence $\ell_V (\Gamma) = (\ell_0,
\ell_1,\dots,\ell_d)$ is the \emph{local $h$-vector} of $\Gamma$ (with respect to
$V$).

The following theorem summarizes some of the main properties of local $h$-vectors
(see Theorems 3.2 and 3.3 and Corollary~4.7 in \cite{Sta92}). For the definition of
regular subdivision we refer the reader to \cite[Definition 5.1]{Sta92}.
  \begin{theorem} {\rm (Stanley~\cite{Sta92})} \label{thm:stalocal}
    \begin{itemize}
      \item[(a)]
        For every simplicial subdivision $\Delta'$ of a pure simplicial complex
        $\Delta$ we have
          \begin{equation} \label{eq:hformula}
            h (\Delta', x) \ = \ \sum_{F \in \Delta} \,
            \ell_F (\Delta'_F, x) \, h (\link_\Delta (F), x).
          \end{equation}
      \item[(b)]
        The local $h$-polynomial $\ell_V (\Gamma, x)$ is symmetric for every simplicial
        subdivision $\Gamma$ of the simplex $2^V$, i.e. we have $\ell_i = \ell_{d-i}$
        for $0 \le i \le d$.
      \item[(c)]
        The local $h$-polynomial $\ell_V (\Gamma, x)$ has nonnegative coefficients for
        every quasi-geometric simplicial subdivision $\Gamma$ of the simplex $2^V$.
      \item[(d)]
        The local $h$-polynomial $\ell_V (\Gamma, x)$ has unimodal coefficients for
        every regular simplicial subdivision $\Gamma$ of the simplex $2^V$.
    \end{itemize}
  \end{theorem}

\section{Proof of the main formula}
\label{sec:proof}

This section gives two proofs of Theorem~\ref{thm:main}, one bijective and
one using generating functions, and deduces Corollary~\ref{cor:main}. As a
byproduct of the second proof, the exponential generating functions of
$f^+_n (x)$ and $f^-_n (x)$ are computed.

\medskip
\noindent
\begin{proof}[First proof of Theorem~\ref{thm:main}] Let us denote by $\cC_n$
the collection of sequences $(\sigma_0, \sigma_1,\dots,\sigma_k)$ of
permutations, where $k \in \NN$ and $\sigma_i \in \mathfrak{S}(S_i)$ for $0
\le i \le k$, such that $(S_0, S_1,\dots,S_k)$ is a weak ordered partition of
$[n]$ with $S_i$ nonempty for $1 \le i \le k$ and $\sigma_0$ is a derangement
of $S_0$. We will describe a one-to-one correspondence $\varphi: \dD^B_n \to
\cC_n$ such that
  \begin{equation} \label{eq:bicondition}
    \iexc_B (w) \ = \ \iexc(\sigma_0) \, + \, \sum_{i=1}^k \, f(\sigma_i) \,
    + \, \lfloor \frac{k+1}{2} \rfloor
  \end{equation}
for every $w \in \dD^B_n$, where $(\sigma_0, \sigma_1,\dots,\sigma_k) = \varphi
(w)$ and $f(\sigma_i)$ stands for $\des(\sigma_i)$ or $\asc(\sigma_i)$, if $i$ is
even or odd, respectively. Given this, using (\ref{eq:dnBidef}) and recalling
that there are ${n \choose r_0, r_1,\dots,r_k}$ weak ordered partitions $(S_0,
S_1,\dots,S_k)$ of $[n]$ satisfying $|S_i| = r_i$ for $0 \le i \le k$, we get
  \begin{eqnarray*}
    d^B_n (x) &=& \sum \ {n \choose r_0, r_1,\dots,r_k} \, x^{\lfloor
    \frac{k+1}{2} \rfloor} \sum_{\sigma_0 \in \dD_{r_0}} x^{\iexc(\sigma_0)} \,
    \left( \, \prod_{i=1}^k \, \sum_{\sigma_i \in \mathfrak{S}_{r_i}}
    x^{\des(\sigma_i)} \right) \\
    & & \\
    &=& \sum \ {n \choose r_0, r_1,\dots,r_k} \, x^{\lfloor
        \frac{k+1}{2} \rfloor} \, d_{r_0}(x) \, A_{r_1}(x) \cdots A_{r_k}(x)
  \end{eqnarray*}
and the proof follows.

To define $\varphi$, consider a derangement $w \in \dD^B_n$ and let $C_1 C_2
\cdots C_m$ be the standard cycle form of $w$. Then there is an index $j \in \{0,
1,\dots,m\}$ such that all elements of $C_1, C_2,\dots,C_j$ are positive and the
first (smallest) element of $C_{j+1}$ is negative. We define $\sigma_0$ as the
product of $C_1, C_2,\dots,C_j$ and $S_0$ as the set of all elements which appear
in these cycles, so that $\sigma_0 \in \mathfrak{S}(S_0)$ is a derangement. The
remaining cycles $C_{j+1},\dots,C_m$ form a word $u$ whose first element is
negative. This word decomposes uniquely as a product $u = u_1 u_2 \cdots u_k$ of
subwords $u_i$ so that for $1 \le i \le k$, all elements of $u_i$ are negative if
$i$ is odd and positive if $i$ is even. We define $S_i$ as the set of absolute
values of the elements of $u_i$ and $\sigma_i \in \mathfrak{S}(S_i)$ as the
permutation which corresponds to the word $u_i$. For instance, if $n = 9$ and $w
= (3 \ 7) (1 \ 4) (-5 \ \, 9 \ -2) (-8 \ -6)$ in standard cycle form, then
$\sigma_0 = (1 \ 4) (3 \ 7)$ in cycle form, $k = 3$ and $\sigma_1 = (5)$,
$\sigma_2 = (9)$, $\sigma_3 = (2, 8, 6)$, as sequences. We set $\varphi(w) =
(\sigma_0, \sigma_1,\dots,\sigma_k)$ and leave it to the reader to verify that
the map $\varphi: \dD^B_n \to \cC_n$ is a well defined bijection.

To verify (\ref{eq:bicondition}) we let $w \in \dD^B_n$ with $\varphi (w) =
(\sigma_0, \sigma_1,\dots,\sigma_k)$ and $u = a_1 a_2 \cdots a_p$ be the word
defined in the previous paragraph. Then, by the definitions of standard cycle
form and (inverse) $B$-excedance, $a \in \Omega_n$ is an inverse $B$-excedance
of $w$ if and only if $a$ is an inverse excedance of $\sigma_0$, or $a = a_i$
for some index $1 \le i < p$ with $a_i > a_{i+1}$, or $a = a_p$. Thus, equation
(\ref{eq:bicondition}) follows.
\end{proof}

For the second proof of Theorem~\ref{thm:main} we set
  \begin{equation} \label{eq:Anexp}
    \aA(t) \ := \ \sum_{n \ge 1} \ A_n (x) \, \frac{t^n}{n!} \ = \ \frac{e^t -
    e^{xt}} {e^{xt} - xe^t}
  \end{equation}
and (see \cite[Proposition 5]{Bre90})
  \begin{equation} \label{eq:dnexp}
    \dD(t) \ := \ \sum_{n \ge 0} \ d_n (x) \, \frac{t^n}{n!} \ = \ \frac{1-x}
    {e^{xt} - xe^t},
  \end{equation}
where $d_0 (x) = 1$. We also recall (see \cite[Theorem 3.3]{CTZ09}
\cite[Theorem~3.2]{Ch09}) that
  \begin{equation} \label{eq:dnBexp}
    \sum_{n \ge 0} \ d^B_n (x) \, \frac{t^n}{n!} \ = \ \frac{(1-x) e^{xt}}
    {e^{2xt} - xe^{2t}},
  \end{equation}
where $d^B_0 (x) = 1$.

\medskip
\noindent
\begin{proof}[Second proof of Theorem~\ref{thm:main}] We denote by $S_n (x)$
(respectively, by $S^+_n (x)$ and $S^-_n (x)$) the right-hand side of
(\ref{eq:main}) (respectively, of (\ref{eq:main+}) and (\ref{eq:main-})), so
that $S_n (x) = S^+_n (x) + S^-_n (x)$ for $n \in \NN$. We compute that

  \begin{eqnarray*}
    \sum_{n \ge 0} \ S^+_n (x) \, \frac{t^n}{n!} &=& \sum_{k, \, r_i \ge 0} \
    x^k \, d_{r_0} (x) \, \frac{t^{r_0}}{r_0!} \ A_{r_1} (x) \,
    \frac{t^{r_1}}{r_1!} \cdots A_{r_{2k}} (x) \, \frac{t^{r_{2k}}}{r_{2k}!} \\
    &=& \sum_{n \ge 0} \ d_n (x) \, \frac{t^n}{n!} \ \, \sum_{k \ge 0} \ x^k
    \left( \sum_{r \ge 1} \ A_r (x) \, \frac{t^r}{r!}  \right)^{2k} \\
    & & \\
    &=& \frac{\dD(t)}{1 - x (\aA(t))^2}
  \end{eqnarray*}
and similarly that

  \begin{eqnarray*}
    \sum_{n \ge 0} \ S^-_n (x) \, \frac{t^n}{n!} &=& \sum_{k, \, r_i \ge 0} \
    x^{k+1} \, d_{r_0} (x) \, \frac{t^{r_0}}{r_0!} \ A_{r_1} (x) \,
    \frac{t^{r_1}}{r_1!} \cdots A_{r_{2k+1}} (x) \, \frac{t^{r_{2k+1}}}{r_{2k+1}!}
    \\ &=& \sum_{n \ge 0} \ d_n (x) \, \frac{t^n}{n!} \ \, \sum_{k \ge 0} \
    x^{k+1} \left( \sum_{r \ge 1} \ A_r (x) \, \frac{t^r}{r!}  \right)^{2k+1} \\
    & & \\
    &=&  \dD(t) \cdot \frac{x \aA(t)}{1 - x (\aA(t))^2}
  \end{eqnarray*}
and conclude that
  $$ \sum_{n \ge 0} \ S_n (x) \, \frac{t^n}{n!} \ = \ \dD(t) \cdot \frac{1 +
     x \aA(t)}{1 - x (\aA(t))^2}. $$

Combining the previous equation with (\ref{eq:Anexp}) and (\ref{eq:dnexp}) we
get, after some straightforward algebraic manipulations, that
  $$ \sum_{n \ge 0} \ S_n (x) \, \frac{t^n}{n!} \ = \ \frac{(1-x) e^{xt}}
    {e^{2xt} - xe^{2t}} \ = \ \sum_{n \ge 0} \ d^B_n (x) \, \frac{t^n}{n!} $$
and the proof follows.
\end{proof}

\medskip
\noindent
\begin{proof}[Proof of Corollary~\ref{cor:main}] As in the second proof of
Theorem~\ref{thm:main}, we denote by $S^+_n (x)$ and $S^-_n (x)$ the right-hand
side of (\ref{eq:main+}) and (\ref{eq:main-}), respectively.

Theorem~\ref{thm:main} shows that $d^B_n (x) = S^+_n (x) + S^-_n (x)$ for every
$n \in \NN$. From the symmetry properties $A_n (x) = x^{n-1} A_n (1/x)$ and $d_n
(x) = x^n \, d_n (1/x)$ of the Eulerian and derangement polynomials for
$\mathfrak{S}_n$ it follows that $S^+_n (x)$ and $S^-_n (x)$ satisfy
(\ref{eq:symf+}) and (\ref{eq:symf-}), respectively. The uniqueness of the
defining properties of $f^+_n (x)$ and $f^-_n (x)$ imply that $f^+_n (x) = S^+_n
(x)$ and $f^-_n (x) = S^-_n (x)$ for every $n \in \NN$. This proves equations
(\ref{eq:main+}) and (\ref{eq:main-}).

The $\gamma$-nonnegativity of $f^+_n (x)$ and $f^-_n (x)$ follows from equations
(\ref{eq:main+}) and (\ref{eq:main-}) and the $\gamma$-nonnegativity of $A_n (x)$
and $d_n (x)$ (see Proposition~\ref{prop:xi+-formula} in the sequel). The last
statement in the corollary follows from (\ref{eq:main++}) and (\ref{eq:main--}).
\end{proof}

Since the polynomials $A_n (x)$ and $d_n (x)$ have nonnegative and symmetric
coefficients and only real roots, we can write
  \begin{equation} \label{eq:Angamma}
    A_n (x) \ = \ (1+x)^{n-1} \ \gamma_n \left( \frac{x}{(1+x)^2} \right)
  \end{equation}
and
  \begin{equation} \label{eq:dngamma}
    d_n (x) \ = \ (1+x)^n \ \xi_n \left( \frac{x}{(1+x)^2} \right)
  \end{equation}
for some polynomials $\gamma_n (x)$ and $\xi_n (x)$ with nonnegative
coefficients. Explicit combinatorial interpretations to these coefficients
are known (see, for instance, \cite[Theorem 5.6]{FSc70} and \cite[Section
4]{AS12}).
Equations (\ref{eq:main+}), (\ref{eq:main-}), (\ref{eq:Angamma}) and
(\ref{eq:dngamma}) imply explicit combinatorial formulas for the polynomials
$\xi^+_n (x) = \sum \xi^+_{n,i} x^i$ and $\xi^-_n (x) = \sum \xi^-_{n,i} x^i$,
appearing in Corollary~\ref{cor:main}, which we record in the following
proposition.

\begin{proposition} \label{prop:xi+-formula}
We have
  \begin{equation} \label{eq:fn+gamma}
    f^+_n (x) \ = \ (1+x)^n \ \xi^+_n \left( \frac{x}{(1+x)^2} \right)
  \end{equation}
and
  \begin{equation} \label{eq:fn-gamma}
    f^-_n (x) \ = \ (1+x)^{n+1} \ \xi^-_n \left( \frac{x}{(1+x)^2} \right),
  \end{equation}
where
  \begin{equation} \label{eq:xin+}
    \xi^+_n (x) \ = \ \sum \ {n \choose r_0, r_1,\dots,r_{2k}} \, x^k \,
      \xi_{r_0} (x) \, \gamma_{r_1} (x) \cdots \gamma_{r_{2k}}(x),
  \end{equation}
  \begin{equation} \label{eq:xin-}
    \xi^-_n (x) \ = \ \sum \ {n \choose r_0, r_1,\dots,r_{2k+1}} \, x^{k+1} \,
      \xi_{r_0}(x) \, \gamma_{r_1} (x) \cdots \gamma_{r_{2k+1}}(x),
  \end{equation}
the sums in the previous equations range as in {\rm (\ref{eq:main+})}
and {\rm (\ref{eq:main-})}, respectively, and $\xi_0 (x) = 1$, $\gamma_0 (x)
= 0$. \qed
\end{proposition}

For the first few values of $n$ we have
  $$ \xi^+_n (x) \ = \ \begin{cases}
    1, \ \ & \text{if \ $n=0$} \\
    0, \ \ & \text{if \ $n=1$} \\
    3x, \ \ & \text{if \ $n=2$} \\
    7x, \ \ & \text{if \ $n=3$} \\
    15x + 57x^2, \ \ & \text{if \ $n=4$} \\
    31x + 458x^2, \ \ & \text{if \ $n=5$} \\
    63x + 2551x^2 + 2763x^3, \ \ & \text{if \ $n=6$} \\
    127x + 12232x^2 + 46861x^3, \ \ & \text{if \ $n=7$}
   \end{cases} $$
and
  $$ \xi^-_n (x) \ = \ \begin{cases}
    0, \ \ & \text{if \ $n=0$} \\
    x, \ \ & \text{if \ $n=1$} \\
    x, \ \ & \text{if \ $n=2$} \\
    x + 11x^2, \ \ & \text{if \ $n=3$} \\
    x + 54x^2, \ \ & \text{if \ $n=4$} \\
    x + 197x^2 + 361x^3, \ \ & \text{if \ $n=5$} \\
    x + 648x^2 + 4379x^3, \ \ & \text{if \ $n=6$} \\
    x + 2039x^2 + 34586x^3 + 24611x^4, \ \ & \text{if \ $n=7$.}
   \end{cases} $$
We are not aware of any combinatorial interpretations for the coefficients of
$\xi^+_n(x)$ or $\xi^-_n (x)$.

The second proof of Theorem~\ref{thm:main} and the proof of
Corollary~\ref{cor:main} yield the following explicit formulas for the exponential
generating functions of $f^+_n (x)$ and $f^-_n (x)$.
\begin{proposition} \label{prop:f+-exp}
We have
  \begin{equation} \label{eq:fn+exp}
    \sum_{n \ge 0} \ f^+_n (x) \, \frac{t^n}{n!} \ = \ \frac{e^{xt} - xe^t}
    {e^{2xt} - xe^{2t}}
  \end{equation}
and
  \begin{equation} \label{eq:fn-exp}
    \sum_{n \ge 0} \ f^-_n (x) \, \frac{t^n}{n!} \ = \ \frac{x(e^t - e^{xt})}
    {e^{2xt} - xe^{2t}}.
  \end{equation}
\end{proposition}
\begin{proof}
We noticed in the proof of Corollary~\ref{cor:main} that $f^+_n (x) = S^+_n (x)$
and $f^-_n (x) = S^-_n (x)$. Thus, the result follows from the formulas in the
second proof of Theorem~\ref{thm:main} and straightforward algebraic manipulations.
\end{proof}

\section{A combinatorial interpretation}
\label{sec:comb}

This section gives a combinatorial interpretation to the coefficients of $f^+_n (x)$
and $f^-_n (x)$ by exploiting the first proof of Theorem~\ref{thm:main}, given in
Section~\ref{sec:proof}.

Consider a signed permutation $w \in \mathfrak{S} (S)$, where $S = \{a_1,
a_2,\dots,a_n\}$ is as in Section~\ref{subsec:signed}. We denote by $m_w$ the minimum
element of $S$ with respect to the natural total order inherited from $\ZZ$ and set
$B^*_n = \{ w \in B_n: w(m_w) > 0\}$.
\begin{proposition} \label{prop:f+-interpret}
  We have
    \begin{equation} \label{eq:fn+interpret}
      f^+_n (x) \ = \ \sum_{w \in \dD^B_n \cap B^*_n} \ x^{\exc_B (w)}
    \end{equation}
  and
    \begin{equation} \label{eq:fn-interpret}
      f^-_n (x) \ = \ \sum_{w \in \dD^B_n \sm B^*_n} \ x^{\exc_B (w)}
    \end{equation}
for every $n \ge 1$.
\end{proposition}
\begin{proof}
We will follow the setup of the first proof of Theorem~\ref{thm:main}. Given $w \in
\dD^B_n$ with $\varphi(w) = (\sigma_0, \sigma_1,\dots,\sigma_k)$, we observe that
$k$ is even if and only if the last element in the standard cycle form of $w$ is
positive. Therefore, equation (\ref{eq:main+}) and the argument in the proof of
Theorem~\ref{thm:main} show that
  $$ f^+_n (x) \ = \ \sum \ x^{\iexc_B (w)}, $$
where the sum ranges over all $w \in \dD^B_n$ for which the last element in the
standard cycle form is positive. Since this element equals $w^{-1} (m_w)$, we get
  $$ f^+_n (x) \ = \ \sum_{w \in \dD^B_n: \ w^{-1} (m_w) > 0} \ x^{\iexc_B (w)}
     \ = \ \sum_{w \in \dD^B_n: \ w (m_w) > 0} \ x^{\iexc_B (w^{-1})}
     \ = \ \sum_{w \in \dD^B_n \cap B^*_n} \ x^{\exc_B (w)}.  $$
Equation~(\ref{eq:fn-interpret}) follows from (\ref{eq:fn+interpret}) and
(\ref{eq:dnBdef}), or by a similar argument.
\end{proof}

\section{The relative local $h$-vector}
\label{sec:relative}

This section reviews the definition of the relative local $h$-polynomial of a
simplicial subdivision of a simplex, introduced in \cite[Section~3]{Ath12} and,
independently (in a different level of generality), in \cite{Ni12}, and establishes
some of its main properties (most of them stated without proof in
\cite[Section~3]{Ath12}). The relative local $h$-polynomial of the barycentric
subdivision of the simplex is also computed (Example~\ref{ex:barycentrel}). This
computation will be used in Section~\ref{sec:geom}.

We will fix a field $\kk$ in this section and work with the notion of a homology
(rather than topological) simplicial subdivision over $\kk$, as in \cite{Ath12}.
Thus, in the definition of a subdivision $\sigma: \Delta' \to \Delta$ we require
that the subcomplex $\Delta'_F := \sigma^{-1} (2^F)$ of $\Delta'$ is a homology
(rather than topological) ball over $\kk$ of dimension $\dim(F)$, for every $F \in
\Delta$; see \cite[Section~2]{Ath12} for details. The following concept was
introduced in \cite[Remark~3.7]{Ath12} and (for regular triangulations of
polytopes) in \cite{Ni12}.

\begin{definition} {\rm (\cite[Section~3]{Ath12})} \label{def:relative}
Let $\Gamma$ be a homology subdivision of a $(d-1)$-dimensional simplex $2^V$, with
subdivision map $\sigma: \Gamma \to 2^V$, and let $E \in \Gamma$. The polynomial
  \begin{equation} \label{eq:deflocalhrel}
    \ell_V (\Gamma, E, x) \ = \sum_{\sigma(E) \subseteq F \subseteq V} \
    (-1)^{d - |F|} \, h (\link_{\Gamma_F} (E), x)
  \end{equation}
is the \emph{relative local $h$-polynomial} of $\Gamma$ (with respect to $V$) at $E$.
\end{definition}

Thus, $\ell_V (\Gamma, E, x)$ reduces to the local $h$-polynomial $\ell_V (\Gamma,
x)$ for $E = \varnothing$.

\begin{example} \label{ex:barycentrel} \rm
Let $\Gamma = \sd(2^V)$ be the barycentric subdivision of an $(n-1)$-dimensional
simplex $2^V$ and $E = \{S_1, S_2,\dots,S_k\}$ be a face of $\Gamma$, where $S_1
\subset S_2 \subset \cdots \subset S_k \subseteq V$ are nonempty sets. We will
show that
  \begin{equation} \label{eq:barycentrel}
    \ell_V (\Gamma, E, x) \ = \ d_{r_0}(x) \, A_{r_1}(x) A_{r_2}(x) \cdots A_{r_k}(x),
  \end{equation}
where $r_0 = |V \sm S_k|$ and $r_i = |S_i \sm S_{i-1}|$ for $1 \le i \le k$ (with
the convention $S_0 = \varnothing$).

We recall from Section~\ref{subsec:sub} that the carrier of $E$ in $\Gamma$ is given
by $\sigma(E) = S_k$. Thus the right-hand side of (\ref{eq:deflocalhrel}) is a sum
over all $S_k \subseteq F \subseteq V$. The restriction $\Gamma_F$ is the
barycentric subdivision of $2^F$ and the link of $E$ in this restriction satisfies
$\link_{\Gamma_F} (E) = \Delta_0 \ast \Delta_1 \ast \cdots \ast \Delta_k$, where
$\Delta_i$ is the simplicial complex of all chains of subsets of $V$ which strictly
contain $S_{i-1}$ and are strictly contained in $S_i$, for $1 \le i \le k$, and
$\Delta_0$ is the simplicial complex of all chains of subsets of $V$ which strictly
contain $S_k$ and are strictly contained in $F$. As a result, we have
  \begin{eqnarray*}
    h(\link_{\Gamma_F} (E), x) &=& h(\Delta_0, x) \, h(\Delta_1, x) \cdots
    h(\Delta_k, x) \\
    &=& A_{|F \sm S_k|}(x) \, A_{r_1}(x) A_{r_2}(x) \cdots A_{r_k}(x).
  \end{eqnarray*}
Multiplying this equation with $(-1)^{d - |F|}$, summing over all $S_k \subseteq
F \subseteq V$ and using (\ref{eq:dnA}) we get (\ref{eq:barycentrel}).
\qed
\end{example}

Our motivation for introducing the relative local $h$-polynomial comes from the
following statement (for another motivation, see \cite[Section~3]{Ni12}).

\begin{proposition} {\rm (\cite[Proposition~3.6]{Ath12})}
                    \label{prop:localrelformula}
  For every homology subdivision $\Gamma$ of the simplex $2^V$ and every homology
  subdivision $\Gamma'$ of $\Gamma$ we have
    \begin{equation} \label{eq:localrelformula}
      \ell_V (\Gamma', x) \ = \, \sum_{E \in \Gamma} \ \ell_E (\Gamma'_E, x) \,
      \ell_V (\Gamma, E, x).
    \end{equation}
\end{proposition}

We now confirm that the polynomial $\ell_V (\Gamma, E, x)$ shares two of the main
properties of $\ell_V (\Gamma, x)$ and deduce a monotonicity property of local
$h$-vectors. These results were stated without proof in \cite[Remark~3.7]{Ath12}.
Here we will sketch the proof, which follows closely ideas of \cite{Sta92} and
their refinements in \cite{Ath10}. For that reason, we will assume familiarity
with the corresponding proofs in \cite{Ath10, Sta92}.

\begin{theorem} \label{thm:relative}
Let $V$ be a set with $d$ elements.
  \begin{itemize}
    \item[(a)]
      The relative local $h$-polynomial $\ell_V (\Gamma, E, x)$ has symmetric
      coefficients, in the sense that
        \begin{equation} \label{eq:relsymm}
          x^{d-|E|} \, \ell_V (\Gamma, E, 1/x) \ = \ \ell_V (\Gamma, E, x),
        \end{equation}
      for every homology subdivision $\Gamma$ of the simplex $2^V$ and every $E \in
      \Gamma$.
    \item[(b)]
      The relative local $h$-polynomial $\ell_V (\Gamma, E, x)$ has nonnegative
      coefficients for every quasi-geometric homology subdivision $\Gamma$ of the
      simplex $2^V$ and every $E \in \Gamma$.
  \end{itemize}
\end{theorem}
\begin{proof}
(a) The proof of \cite[Theorem 4.2]{Ath10} can be adapted as follows. Using the
defining equation (\ref{eq:deflocalhrel}) and \cite[Proposition~2.1]{Ath10}, we find
that
  \begin{eqnarray*}
    x^{d-|E|} \, \ell_V (\Gamma, E, 1/x) &=& \sum_{\sigma(E) \subseteq F \subseteq
    V} (-1)^{d - |F|} \, x^{d-|E|} \, h (\link_{\Gamma_F} (E), 1/x) \\
          & & \\
    &=& \sum_{\sigma(E) \subseteq F \subseteq V} (-x)^{d - |F|} \, h
    (\inte(\link_{\Gamma_F} (E)), x).
  \end{eqnarray*}
An inclusion-exclusion argument, similar to the one in the proof of
\cite[(4.3)]{Ath10}, shows that
  $$ h (\inte(\link_{\Gamma_F} (E)), x) \ = \ \sum_{\sigma(E) \subseteq G
     \subseteq F} (x-1)^{|F|-|G|} \, h (\link_{\Gamma_G} (E), x). $$

Replacing $h (\inte(\link_{\Gamma_F} (E)), x)$ in the first formula by the right-hand
side of the previous equation and changing the order of summation, as in the proof of
\cite[Theorem 4.2]{Ath10}, results in (\ref{eq:relsymm}).

\medskip
(b) The special case $E = \varnothing$ is equivalent to part (iii) of
\cite[Theorem~3.3]{Ath12} (essentially, part (c) of Theorem~\ref{thm:stalocal}). The
general case follows by the argument in the proof of \cite[Theorem~5.1]{Ath10}
(generalizing that in the proof of \cite[Theorem~4.6]{Sta92}), where the role of
$\Delta$ in that proof is played by $\link_\Gamma (E)$, the role of $d$ is played by
$d - |E| = \dim \link_\Gamma (E) + 1$ and the role of $e$ is played by the rank $d -
|\sigma(E)|$ of the interval $[\sigma(E), V]$ in the lattice of subsets of $V$.
\end{proof}

For polynomials $p(x), q(x) \in \RR[x]$ we write $p(x) \ge q(x)$ if the difference
$p(x) - q(x)$ has nonnegative coefficients.

\begin{corollary} \label{cor:localmonotone}
For every quasi-geometric homology subdivision $\Gamma$ of the simplex $2^V$ and
every quasi-geometric homology subdivision $\Gamma'$ of $\Gamma$, we have $\ell_V
(\Gamma', x) \ge \ell_V (\Gamma, x)$.
\end{corollary}
\begin{proof}
The right-hand side of (\ref{eq:localrelformula}) reduces to $\ell_V(\Gamma, x)$
for $E = \varnothing$. The other terms in the sum are nonnegative by
Theorems~\ref{thm:stalocal} (c) and \ref{thm:relative} (b) and the proof follows.
\end{proof}

\section{A geometric interpretation}
\label{sec:geom}

This section formally defines the simplicial subdivision $K_n$ and gives two proofs
of Theorem~\ref{thm:localint}, one using the theory of (relative) local $h$-vectors
(specifically, Proposition~\ref{prop:localrelformula}) and another using generating
functions.

Let $\Delta$ be a simplicial complex. The \emph{cubical barycentric subdivision} (see,
for instance, \cite[Section~2.3]{BBC97}) of $\Delta$, denoted
$\sd_c (\Delta)$, is defined as the set of all nonempty closed intervals $[F, G]$ in
the face poset $\fF(\Delta)$, partially ordered by inclusion. It follows from
\cite[Theorem~6.1~(a)]{Wa88} and \cite[Equation~(3.24)]{StaEC1} that the order complex,
say $\Delta'$, of $\sd_c (\Delta)$ is homeomorphic to $\Delta$. Moreover, $\Delta'$
is naturally a simplicial subdivision of $\Delta$: the carrier of a face of $\Delta'$
is the maximum element of the largest of the intervals in the corresponding chain of
intervals of $\fF(\Delta)$. We will denote by $K_n$ the order complex of $\sd_c
(2^{[n]})$, so that $K_n$ is a simplicial subdivision of the simplex $2^{[n]}$ (see
Figure~\ref{fig:K3} for the case $n=3$). We note that $K_n$ is the special case
$N=1$ of a subdivision of the simplex considered in \cite[p.~414]{CMS84}.

The following statement is an essential step for both proofs of
Theorem~\ref{thm:localint} which will be given in this section.

\begin{proposition} \label{prop:Knhpoly}
We have $h (K_n, x) = B^+_n (x)$ for $n \in \NN$.
\end{proposition}
\begin{proof}
The poset $\sd_c (2^{[n]})$ consists of all intervals of the form $[A, B]$, where
$\varnothing \neq A \subseteq B \subseteq [n]$, partially ordered by inclusion. To
describe this poset differently, we consider the following poset $(P_n, \preceq)$.
The elements of $P_n$ are the subsets of $\Omega_n$ which contain at least one
positive number and at most one number from each set $\{i , -i\}$ for $i \in \{1,
2,\dots,n \}$; the partial order is reverse inclusion. We observe that the map
$\varphi: \sd_c (2^{[n]}) \to P_n$ defined by $\varphi([A, B]) = \ A \, \cup \,
(-([n] \sm B) )$ is a poset isomorphism. Thus, we may identify $K_n$ with the order
complex of $P_n$.

For $S = \{s_1, s_2,\dots,s_k\} \subseteq [n]$ with $s_1 < s_2 < \cdots < s_k$, we
define $\alpha_{P_n}(S)$ as the number of chains $F_1 \prec F_2 \prec \cdots \prec
F_k$ in $P_n$ such that $|F_i| = s_{k-i+1}$ for $i \in \{1, 2,\dots,k\}$. The map
$\alpha_{P_n}: 2^{[n]} \rightarrow \NN$ is the flag $f$-vector of $P_n$; see
\cite[Section~3.13]{StaEC1}. The chains of $P_n$ enumerated by $\alpha_{P_n}(S)$
are in one-to-one correspondence with the elements $w \in B^+_n$ for which
$\Des_B(w) \subseteq n-S := \{ n-s: \, s \in S\}$. Indeed, given such a chain, the
corresponding element of $B^+_n$ consists of the elements of $[n] \sm \{|s|: s
\in F_1\}$ in increasing order, followed by those of $F_1 \sm F_2$ in increasing
order and so on, followed at the end by the elements of $F_k$ in increasing order.

Recall that the flag $h$-vector $\beta_{P_n}:2^{[n]} \rightarrow \ZZ$ of $P_n$ is
defined by
  $$ \beta_{P_n}(S) \ = \ \sum_{T \subseteq S} (-1)^{|S \sm T|}
     \ \alpha_{P_n}(T), $$
for $S \subseteq [n]$, or equivalently, by
  $$ \alpha_{P_n}(S) \ = \ \sum_{T\subseteq S} \beta_{P_n}(T) $$
for $S\subseteq [n]$. Since $\alpha_{P_n}(S)$ enumerates signed permutations $w
\in B^+_n$ for which $\Des_B(w) \subseteq n -S$, by the Principle of
Inclusion-Exclusion we get that $\beta_{P_n}(S)$ enumerates signed permutation
$w \in B_n^+$ for which $\Des_B(w) = n-S$. The result follows from this
interpretation by recalling \cite[Section~3.13]{StaEC1} that
  $$ h_k(K_n) \ = \ \sum_{S\subseteq [n], |S|=k} \beta_{P_n}(S) $$
and switching $S$ to $n-S$ in the previous equation.
\end{proof}

Our first proof of Theorem~\ref{thm:localint} will be based on the fact that
$K_n$ can be viewed as a subdivision of the barycentric subdivision $\sd
(2^{[n]})$. To explain how, we consider the following setup. Let $V = \{v_1,
v_2,\dots,v_d\}$ be a set totally ordered by $v_1 < v_2 < \cdots < v_d$. We
recall that $\sd_c (V)$ denotes the poset of intervals in $V$ of the form
$[v_i, v_j] = \{v_i, v_{i+1},\dots,v_j\}$ for $1 \le i \le j \le n$, partially
ordered by inclusion. We denote by $\Gamma$ the order complex of $\sd_c (V)$,
consisting of all chains of such intervals. For such a chain $G \in \Gamma$,
we define $\sigma(G)$ as the set of all endpoints of the intervals in $G$.
Thus we have a well defined map $\sigma: \Gamma \to 2^V$.

  \begin{figure}
  \epsfysize = 2.0 in \centerline{\epsffile{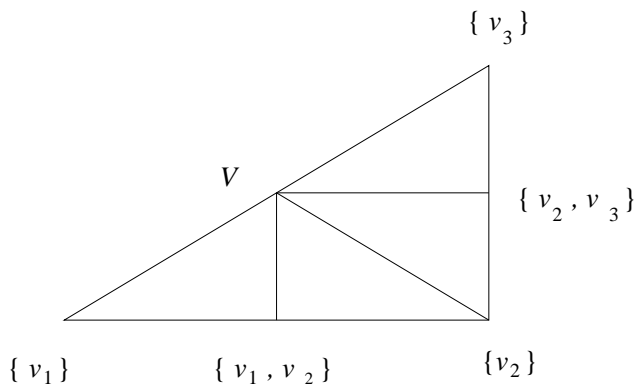}}
  \caption{The subdivision $\Gamma$ for $d=3$}
  \label{fig:Int3}
  \end{figure}

\begin{lemma} \label{lem:Knbary}
Under the previous assumptions and notation, the map $\sigma: \Gamma \to 2^V$
turns $\Gamma$ into a geometric subdivision of $2^V$. The number of facets of
$\Gamma$ is equal to $2^{d-1}$, where $d$ is the number of elements of $V$.
\end{lemma}
\begin{proof}
Let $\Sigma_V$ be a geometric $(d-1)$-dimensional simplex whose vertices are
labeled by the singleton subsets of $V = \{v_1, v_2,\dots,v_d\}$. We will
construct a geometric simplicial subdivision (triangulation) $\Gamma_V$ of
$\Sigma_V$ whose vertices are labeled (in a one-to-one fashion) with the
closed intervals in the total order $V$, so that: (a) the singleton intervals
label the vertices of $\Sigma_V$; (b) the point labeled by a non-singleton
interval $I = [v_i, v_j] \in \sd_c (V)$ lies in the relative interior of the
edge of $\Sigma_V$ whose endpoints are labeled by $\{v_i\}$ and $\{v_j\}$; and
(c) the faces of $\Gamma_V$ correspond to the chains of intervals (see
Figure~\ref{fig:Int3} for the case $d=3$).

We proceed by induction on $d$. The triangulation $\Gamma_V$ is a single
point for $d=1$ and the triangulation of a line segment with one interior
point (labeled by $\{v_1, v_2\}$) for $d=2$. We assume $d \ge 3$ and
set $U = V \sm \{v_d\}$ and $W = V \sm \{v_1\}$. We choose the simplices
$\Sigma_U$ and $\Sigma_W$ as the codimension one faces of $\Sigma_V$ which
correspond to $U$ and $V$ and, using the inductive hypothesis, triangulations
$\Gamma_U$ and $\Gamma_W$ of these two simplices having properties (a), (b)
and (c) with respect to the totally ordered subsets $U$ and $W$ of $V$,
respectively. Clearly, we may choose these triangulations to have the same
restriction on the face $\Sigma_U \cap \Sigma_W$ of $\Sigma_V$. We then
label by $V$ an arbitrary point $p$ in the relative interior of the edge
of $\Sigma_V$ whose endpoints are labeled with $\{v_1\}$ and $\{v_d\}$
and define $\Gamma_V$ as the collection consisting of all simplices in
$\Gamma_U \cup \Gamma_W$ and the cones of these on the vertex $p$. We leave
it to the reader to verify that $\Gamma_V$ has properties (a), (b) and (c)
and that it realizes an abstract simplicial subdivision of $2^V$ with the
required properties.
\end{proof}

We now recall that $K_n$ consists of all chains of intervals of the form
$[A, B]$, where $\varnothing \ne A \subseteq B \subseteq [n]$. We define
the carrier of such a chain $\cC$ as the set of all endpoints of the intervals
in $\cC$ and note that this set is a chain in the poset $\fF(2^{[n]})$ of
nonempty subsets of $[n]$ and hence belongs to the barycentric subdivision
$\sd (2^{[n]})$. Applying Lemma~\ref{lem:Knbary} to an arbitrary chain $V \in
\sd (2^{[n]})$ we conclude that $K_n$ is a subdivision of $\sd (2^{[n]})$ and
that the restriction of this subdivision to a nonempty face $V \in \sd
(2^{[n]})$ of dimension $d-1$ has exactly $2^{d-1}$ facets.

\begin{lemma} \label{lem:2^n-1facets}
Let $\Gamma$ be a quasi-geometric simplicial subdivision of a $(d-1)$-dimensional
simplex $2^V$. If the restriction $\Gamma_F$ has exactly $2^{\dim(F)}$ facets
for every nonempty face $F$ of $2^V$, then
  $$ \ell_V (\Gamma, x) \ = \ \begin{cases}
                                  x^{d/2}, & \text{if $d$ is even} \\
                                  0, & \text{if $d$ is odd.}
                                  \end{cases} $$
\end{lemma}
\begin{proof}
We recall that the number of facets of a simplicial complex $\Delta$ is equal
to the value of the $h$-polynomial $h(\Delta, x)$ at $x=1$. Thus, setting $x=1$
in the defining equation (\ref{eq:deflocalh}) and using the assumption on
$\Gamma$, we find that
  $$ \ell_V (\Gamma, 1) \ = \ (-1)^d \, + \, \sum_{k = 1}^d \ (-1)^{d - k} {d
     \choose k} \, 2^{k-1} \ = \ \begin{cases}
                                  1, & \text{if $d$ is even} \\
                                  0, & \text{if $d$ is odd}
                                  \end{cases} $$
and the result follows from parts (b) and (c) of Theorem~\ref{thm:stalocal}.
\end{proof}

\medskip
\noindent
\begin{proof}[First proof of Theorem~\ref{thm:localint}] Let us denote by
$\ell^+_n (x)$ the local $h$-polynomial of $K_n$. To compute this polynomial,
we will apply Proposition~\ref{prop:localrelformula} to $\Gamma' = K_n$ and
$\Gamma = \sd (2^{[n]})$. Let $E = \{S_1, S_2,\dots,S_k\}$ be a face of
$\Gamma$ with $k$ elements, where $S_1 \subset S_2 \subset \cdots \subset
S_k \subseteq [n]$ are nonempty sets. We have already noted that the
restriction $\Gamma'_E$ satisfies the assumptions of
Lemma~\ref{lem:2^n-1facets}. Thus, by Lemma~\ref{lem:2^n-1facets} we have
  \begin{equation} \label{eq:2^k-1facets}
    \ell_E (\Gamma'_E, x) \ = \ \begin{cases}
                                  x^{k/2}, & \text{if $k$ is even} \\
                                  0, & \text{if $k$ is odd.}
                                  \end{cases}
  \end{equation}
The relative local $h$-vector of $\Gamma$ was computed in
Example~\ref{ex:barycentrel}. Thus, in view of (\ref{eq:2^k-1facets}) and
(\ref{eq:barycentrel}), Proposition~\ref{prop:localrelformula} yields that
  $$ \ell^+_n (x) \ = \ \sum \ {n \choose r_0, r_1,\dots,r_k} \, x^{k/2} \,
     d_{r_0} (x) \, A_{r_1} (x) \cdots A_{r_k}(x), $$
where the sum ranges over all even numbers $k \in \NN$ and over all sequences
$(r_0, r_1,\dots,r_k)$ of nonnegative integers which sum to $n$. This equation
and (\ref{eq:main+}) imply that $\ell^+_n (x) = f^+_n (x)$ and the first
statement of Theorem~\ref{thm:localint} follows.

We leave to the reader to verify that $K_n$ can be obtained from the trivial
subdivision of the simplex by successive stellar subdivisions. This implies
that $K_n$ is a regular subdivision. The claim that $f^+_n (x)$ has nonnegative,
symmetric and unimodal coefficients follows from the main properties of local
$h$-polynomials \cite{Sta92} (see Theorem~\ref{thm:stalocal}).
Equation~(\ref{eq:localint+}) follows from the fact that $f^+_n (x) = \ell^+_n
(x)$, the defining equation (\ref{eq:deflocalh}) of local $h$-polynomials and
Proposition~\ref{prop:Knhpoly}. Given that $d^B_n (x) = f^+_n (x) + f^-_n (x)$
and $B_n (x) = B^+_n (x) + B^-_n (x)$ for every $n$, equation (\ref{eq:localint-})
is a consequence of (\ref{eq:dnB}) and (\ref{eq:localint+}).
\end{proof}

For the second proof of Theorem~\ref{thm:localint} we will need the exponential
generating functions of $B^+_n (x)$ and $B^-_n (x)$. These will be computed in
Section~\ref{sec:half}.

\medskip
\noindent
\begin{proof}[Second proof of Theorem~\ref{thm:localint}] Let us denote by
$\ell^+_n (x)$ and $\ell^-_n (x)$ the right-hand side of (\ref{eq:localint+})
and (\ref{eq:localint-}), respectively. Proposition~\ref{prop:Knhpoly} and
(\ref{eq:deflocalh}) imply that $\ell^+_n (x)$ is equal to the local
$h$-polynomial of $K_n$. Thus, we need to show that $\ell^+_n (x) = f^+_n
(x)$ and $\ell^-_n (x) = f^-_n (x)$ for every $n$. From the definition of
$\ell^+_n (x)$ and $\ell^-_n (x)$ and Proposition~\ref{prop:B+-nexp} we
get
  $$ \sum_{n \ge 0} \ \ell^+_n (x) \, \frac{t^n}{n!} \ = \ e^{-t} \
     \sum_{n \ge 0} \ B^+_n (x) \, \frac{t^n}{n!} \ = \ \frac{e^{xt} - xe^t}
     {e^{2xt} - xe^{2t}} $$
and
  $$ \sum_{n \ge 0} \ \ell^-_n (x) \, \frac{t^n}{n!} \ = \ e^{-t} \
     \sum_{n \ge 0} \ B^-_n (x) \, \frac{t^n}{n!} \ = \ \frac{x(e^t - e^{xt})}
     {e^{2xt} - xe^{2t}}. $$
The result follows from these equations and Proposition~\ref{prop:f+-exp}.
\end{proof}

\section{A decomposition of the Eulerian polynomial of type $B$}
\label{sec:half}

This section studies the decomposition of the Eulerian polynomial $B_n (x)$ as
a sum of $B^+_n (x)$ and $B^-_n (x)$. First, it is observed that a simple
relation between the two summands holds. Then, using the theory of local
$h$-vectors and results of Section~\ref{sec:geom}, a simple formula for $B^+_n
(x)$ in terms of the Eulerian polynomial $A_n (x)$ is proven
(Proposition~\ref{prop:Bn+formula}). From this formula, it is deduced that
$B^+_n (x)$ and $B^-_n (x)$ are real-rooted (Corollary~\ref{cor:B+-nrealroots}),
hence unimodal and log-concave, and a new proof of the unimodality of $B_n (x)$
is derived. Finally, recurrences and generating functions for $B^+_n (x)$ and
$B^-_n (x)$ are given. These lead to recurrences and generating functions for
$f^+_n (x)$ and $f^-_n (x)$ and to yet another proof of Theorem~\ref{thm:localint}.

We recall that $B^+_n (x)$ and $B^-_n (x)$ are defined by (\ref{eq:Bn+def})
and (\ref{eq:Bn-def}). For the first few values of $n$ we have
  $$ B^+_n(x) \ = \ \begin{cases}
    1, & \ \text{if \ $n=0$} \\
    1, & \ \text{if \ $n=1$} \\
    1 + 3x, & \ \text{if \ $n=2$} \\
    1 + 16x + 7x^2, & \ \text{if \ $n=3$} \\
    1 + 61x + 115x^2 + 15x^3, & \ \text{if \ $n=4$} \\
    1 + 206x + 1056x^2 + 626x^3 + 31x^4, & \ \text{if \ $n=5$} \\
    1 + 659x + 7554x^2 + 11774x^3 + 2989x^4 + 63x^5, & \ \text{if \ $n=6$}
  \end{cases} $$
and
  $$ B^-_n(x) \ = \ \begin{cases}
    0, & \ \text{if \ $n=0$} \\
    x, & \ \text{if \ $n=1$} \\
    3x + x^2, & \ \text{if \ $n=2$} \\
    7x + 16x^2 + x^3, & \ \text{if \ $n=3$} \\
    15x + 115x^2 + 61x^3 + x^4, & \ \text{if \ $n=4$} \\
    31x + 626x^2 + 1056x^3 + 206x^4 + x^5, & \ \text{if \ $n=5$} \\
    63x + 2989x^2 + 11774x^3 + 7554x^4 + 659x^5 + x^6, & \ \text{if \ $n=6$.}
  \end{cases} $$

\medskip
The previous data suggest the following statement.

\begin{lemma} \label{lem:rec}
We have $B^-_n (x) = x^n B^+_n (1/x)$ for $n \ge 1$.
\end{lemma}
\begin{proof}
Given a signed permutation $w = (w(a_1), w(a_2),\dots,w(a_n)) \in B_n$, where
the notation is as in Section~\ref{subsec:signed}, we set $-w := (-w(-a_1),
-w(-a_2),\dots,-w(-a_n)) \in B_n$. Then the induced map $\varphi: B^+_n \to
B^-_n$ defined by $\varphi(w) = -w$ is a bijection. Moreover, for every $w \in
B^+_n$, an index $i \in \{0, 1,\dots,n-1\}$ is a $B$-ascent of $w$ if and
only if $i$ is a $B$-descent of $\varphi(w)$ and the proof follows.
\end{proof}

To prove the formula for $B^+_n (x)$ promised, we will use the construction
of the $r$th edgewise subdivision $\Delta^{\langle r \rangle}$ of a simplicial
complex $\Delta$. We refer the reader to \cite{BR05, BW09} for the definition and
history of this subdivision and recall the following known facts. First, the
restriction $\Delta^{\langle r \rangle}_F$ of $\Delta^{\langle r \rangle}$ has
exactly $r^{\dim(F)}$ facets for every nonempty face $F \in \Delta$. Second,
combining \cite[Corollary~6.8]{BR05} with \cite[Corollary~1.2]{BW09}, one gets
the explicit formula
  \begin{equation} \label{eq:hedgewise}
    h(\Delta^{\langle r \rangle}, x) \ = \ {\rm E}_r \left( (1 + x + \cdots
    + x^{r-1})^d \, h (\Delta, x) \right)
  \end{equation}
for the $h$-polynomial of $\Delta^{\langle r \rangle}$, where $d-1$ is the
dimension of $\Delta$ and ${\rm E}_r$ is the operator on polynomials (more
generally, on formal power series) defined by
\[ {\rm E}_r \left( \, \sum_{k \ge 0} \, c_k x^k \right) \ = \ \sum_{k \ge 0} \,
    c_{rk} x^k \ = \ c_0 + c_r x + c_{2r} x^2 + \cdots. \]
Figure~\ref{fig:KK3} shows the second edgewise subdivision of the barycentric
subdivision of the 2-dimensional simplex.

  \begin{figure}
  \epsfysize = 1.5 in \centerline{\epsffile{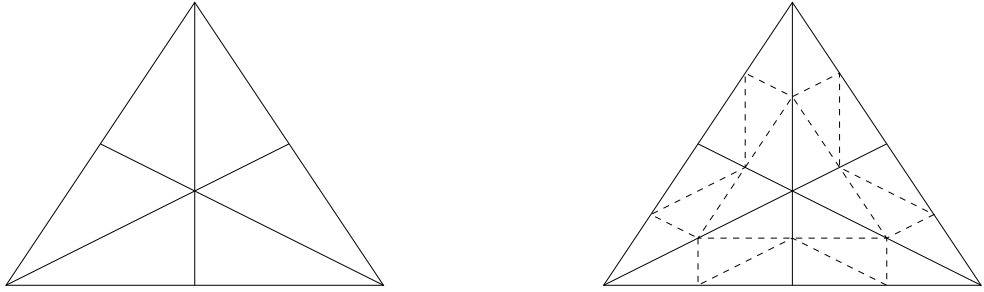}}
  \caption{The barycentric subdivision of the 2-simplex and its second
  edgewise subdivision $K'_3$}
  \label{fig:KK3}
  \end{figure}

\begin{proposition} \label{prop:Bn+formula}
We have $B^+_n (x) = {\rm E}_2 \left( (1 + x)^n A_n (x) \right)$ for every
$n \ge 1$.
\end{proposition}
\begin{proof}
We consider the subdivision $K_n$ and the second edgewise subdivision $K'_n$
of the barycentric subdivision $\sd (2^{[n]})$ (see Figures~\ref{fig:K3} and
\ref{fig:KK3} for the special case $n=3$). Applying (\ref{eq:hformula}) for
$\Delta' = K_n$ or $K'_n$, respectively, and $\Delta = \sd (2^{[n]})$ we get

  \begin{eqnarray*}
    h(K_n, x) &=& \sum_{F \in \Delta} \ \ell_F ((K_n)_F, x) \, h
                  (\link_\Delta (F), x), \\
    h(K'_n, x) &=& \sum_{F \in \Delta} \ \ell_F ((K'_n)_F, x) \, h
                  (\link_\Delta (F), x).
  \end{eqnarray*}
Since both restrictions $(K_n)_F$ and $(K'_n)_F$ have exactly $2^{\dim(F)}$
facets for every nonempty face $F \in \Delta$, it follows from the previous
formulas and Lemma~\ref{lem:2^n-1facets} that $h(K_n, x) = h(K'_n, x)$.
Combining this equality with Proposition~\ref{prop:Knhpoly} we get $B^+_n
(x) = h(K'_n, x)$ for every $n \ge 1$. Formula (\ref{eq:hedgewise}) implies
that
  $$ B^+_n (x) \ = \ {\rm E}_2 \left( (1 + x)^n h (\Delta, x) \right) \ = \
     {\rm E}_2 \left( (1 + x)^n A_n (x) \right) $$
for $n \ge 1$ and the proof follows.
\end{proof}

\begin{remark} \label{rem:cry} \rm
We thank Mirk\'{o} Visontai for informing us that a formula similar to the one 
in Proposition~\ref{prop:Bn+formula} can be derived from \cite[Theorem~4.4]{ABR01}, 
for which a bijective proof was given in \cite{LP11}.
\end{remark}

We will use the following lemma to deduce the real-rootedness of $B^+_n (x)$
and $B^-_n (x)$.

\begin{lemma} \label{lem:Er}
Let $p(x)$ be a polynomial with real coefficients and let $r$ be a positive
integer.
  \begin{itemize}
    \item[(a)]
      If $p(x)$ has unimodal coefficients, then so does ${\rm E}_r (p(x))$.
    \item[(b)]
      If $p(x)$ has nonnegative and log-concave coefficients, with no internal
      zeros, then so does ${\rm E}_r (p(x))$.
    \item[(c)]
      If $p(x)$ is real-rooted, then so is ${\rm E}_r (p(x))$.
  \end{itemize}
\end{lemma}
\begin{proof}
Part (a) is trivial and part (b) can be left as an excercise. For part (c)
we set $p(x) = \sum_{k \ge 0} a_k x^k$ and note that the matrix
$(a_{ri - rj})_{i,j=0}^\infty$ is a submatrix of $(a_{i - j})_{i,j=0}^\infty$.
Therefore, every minor of the former is also a minor of the latter and the
result follows from Theorem~\ref{thm:realroots}.
\end{proof}

\begin{corollary} \label{cor:B+-nrealroots}
The polynomials $B^+_n (x)$ and $B^-_n (x)$ are real-rooted for every $n \ge
1$. They are unimodal with peaks at $\lfloor n/2 \rfloor$ and $\lfloor (n+1)/2
\rfloor$, respectively, for every $n \ge 2$.
\end{corollary}
\begin{proof}
The first statement follows from Lemma~\ref{lem:rec},
Proposition~\ref{prop:Bn+formula} and the fact that the Eulerian polynomial
$A_n(x)$ is real-rooted for $n \ge 1$, via part (c) of Lemma~\ref{lem:Er}.
The second statement follows from Lemma~\ref{lem:rec},
Proposition~\ref{prop:Bn+formula} and the fact that $(1 + x)^n A_n (x)$ is
a polynomial of degree $2n-1$ with symmetric and unimodal coefficients, via
part (a) of Lemma~\ref{lem:Er}.
\end{proof}

\begin{remark} \label{rem:Bnformula} \rm
Since $B_n (x) = B^+_n (x) + B^-_n(x)$, Lemma~\ref{lem:rec} and
Proposition~\ref{prop:Bn+formula} express the Eulerian polynomial $B_n (x)$
as a sum of two unimodal polynomials with peaks which differ by at most one
(see Corollary~\ref{cor:B+-nrealroots}). This decomposition shows that the
unimodality of $B_n (x)$ is a consequence of the unimodality of $A_n (x)$.
For a $\gamma$-nonnegativity proof of the unimodality of $B_n (x)$, see
\cite[Proposition~4.16]{Pe07}. For an equation relating the Eulerian
polynomials of types $A$, $B$ and $D$, see \cite[Lemma~9.1]{Ste94}.
\end{remark}

We will now give recurrences and generating functions for $B^+_n (x)$ and
$B^-_n (x)$.

\begin{proposition} \label{prop:B+-nexp}
We have
  \begin{equation}\label{eq:B+nrec}
    B^+_n(x) \ = \ 2(n-1)x\, B^+_{n-1}(x) \, + \, 2x(1-x)\, \frac{\partial
                   B^+_{n-1}}{\partial x} (x) \, + \, B_{n-1}(x)
  \end{equation}
for every $n \ge 1$,
  \begin{equation}\label{eq:B+nexp}
    \sum_{n \ge 0} B^+_n(x) \frac{t^n}{n!} \ = \ \frac{e^t (e^{xt} - xe^t)}
                                                 {e^{2xt} - xe^{2t}}
\end{equation}
and
  \begin{equation}\label{eq:B-nexp}
    \sum_{n \ge 0} B^-_n(x) \frac{t^n}{n!} \ = \ \frac{xe^t (e^t - e^{xt})}
                                                 {e^{2xt} - xe^{2t}}.
\end{equation}
\end{proposition}
\begin{proof}
Let $w = u_1 u_2 \cdots u_{n-1} \in B_{n-1}$ be a signed permutation,
represented as a word. For $i \in \{1, 2,\dots,n\}$, we will denote by $w_i$
(respectively, $w_{-i}$) the signed permutation in $B_n$ obtained from $w$
by inserting $n$ (respectively, $-n$) between $u_{i-1}$ and $u_i$. For $1
\le i \le n-1$ we have $w_i \in B^+_n$ (respectively, $w_{-i} \in B^+_n$) if
and only if $w \in B^+_{n-1}$. On the other hand, $w_n \in B^+_n$ and $w_{-n}
\in B^-_n$ for every $w \in B_{n-1}$. Moreover, for $1 \le i \le n-1$ we have
\begin{equation}
  \des_B(w_{\pm i}) \ = \ \begin{cases}
  \des_B(w), & \textrm{if $i-1 \in \Des_B(w)$}\\
  \des_B(w) + 1, & \textrm{if $i-1 \notin \Des_B(w)$}
  \end{cases} \nonumber
\end{equation}
and $\des_B(w_n) \ = \ \des_B(w)$. Thus, we compute that

\begin{eqnarray*}
B^+_n(x) & = & \sum_{\sigma \in B^+_n} x^{\des_B(\sigma)} \ = \ \sum_{i=1}^{n-1}
\left(\sum_{w \in B^+_{n-1}} x^{\des_B(w_i)} \, + \, x^{\des_B (w_{-i}) } \right)
+ \sum_{w\in B_{n-1}} x^{\des_B(w_n)} \\
& & \\
& = & 2 \sum_{w \in B^+_{n-1}} \left( \des_B(w) \, x^{\des_B(w)} + (n-1-\des_B(w))
\, x^{\des_B(w) + 1} \right) \ + \ B_{n-1}(x) \\
& & \\
& = & 2(n-1) \sum_{w \in B^+_{n-1}} x^{\des_B(w) + 1} \ + \ 2 (1-x) \sum_{w \in
B^+_{n-1}} \des_B(w) \, x^{\des_B(w)} \ + \ B_{n-1}(x) \\
& & \\
& = & 2(n-1) x \, B^+_{n-1}(x) \ + \ 2x(1-x)\, \frac{\partial B^+_{n-1}}{\partial
x}(x) \ + \ B_{n-1}(x),
\end{eqnarray*}

\medskip
\noindent
which proves (\ref{eq:B+nrec}). We now claim that
  \begin{equation} \label{eq:ratB+n}
    \frac{B^+_n(x)}{(1-x)^n} \ = \ \sum_{i \ge 0} \ \left( (2i+1)^n - (2i)^n \right)
    \, x^i.
  \end{equation}
Given that $B^+_0 (x) = 1$, equation (\ref{eq:B+nexp}) then follows by
straightforward computations. To prove (\ref{eq:ratB+n}), denote by $a_n(i)$ the
coefficient of $x^i$ in the expansion of $B^+_n(x)/(1-x)^n$ as a formal power
series. Dividing (\ref{eq:B+nrec}) by $(1-x)^n$ and using the equality
  $$ \frac{\partial}{\partial x} \left( \frac{B^+_{n-1}(x)}{(1-x)^{n-1}} \right)
     \ = \ \frac{{\displaystyle \frac{\partial B^+_{n-1}} {\partial x}(x)}}
     {(1-x)^{n-1}} \ + \ (n-1) \, \frac{B^+_{n-1}(x)}{(1-x)^{n}} $$
we find that
  $$ \frac{B^+_n(x)}{(1-x)^n} \ = \ 2x \ \frac{\partial}{\partial x} \left(
     \frac{B^+_{n-1}(x)}{(1-x)^{n-1}}\right) \ + \ \frac{B_{n-1}(x)}{(1-x)^n}.
  $$
Comparing the coefficients of $x^i$ in the two sides of the previous equation and
using \cite[Theorem~3.4~(ii)]{Bre94}, we get $a_n(i) = 2i a_{n-1}(i) + (2i+1)^{n-1}$.
The claim then follows by induction on $n$.

Equation (\ref{eq:B-nexp}) follows from (\ref{eq:B+nexp}) and Lemma~\ref{lem:rec}.
Alternatively, it follows from (\ref{eq:B-nexp}) and the formula for the
exponential generating function of $B_n (x)$ \cite[Theorem~3.4~(iv)]{Bre94}.
\end{proof}

We now deduce recurrence relations for $f^+_n (x)$ and $f^-_n (x)$.

\begin{proposition} \label{prop:recf+n}
For $n \ge 1$ we have
  $$ f^+_n (x) \ = \ (2(n-1)x-1)\, f^+_{n-1}(x) \, + \, 2x(1-x) \, \frac{\partial
     f^+_{n-1}} {\partial x}(x) \, + \, 2(n-1) x\, f^+_{n-2}(x) \, + \, d^B_{n-1} (x).
  $$
\end{proposition}
\begin{proof}
Using equation (\ref{eq:localint+}), we compute that
  \begin{eqnarray*}
    f^+_n (x) & = & \sum_{k=0}^n \, (-1)^{n-k} \binom{n}{k} \ B^+_k (x) \\
    & & \\
    &=& \sum_{k=1}^n \, (-1)^{n-k} \binom{n-1}{k-1} \ B^+_k (x) \ + \ \sum_{k=0}^{n-1}
    \, (-1)^{n-k} \binom{n-1}{k} \ B^+_k (x) \\
    & & \\
    &=& \sum_{k=1}^n \, (-1)^{n-k} \binom{n-1}{k-1} \ B^+_k(x) \ - \ f^+_{n-1} (x).
  \end{eqnarray*}
Substituting for $B^+_k(x)$ the right-hand side of (\ref{eq:B+nrec}), setting
  $$ S_n(x) \ = \ 2x \, \sum_{k=1}^n \, (-1)^{n-k} \, (k-1) \binom{n-1}{k-1} \
     B^+_{k-1}(x) $$
and using (\ref{eq:dnB}), we get
  \begin{eqnarray*}
    f^+_n(x) & = & S_n(x) \, + \, 2x(1-x) \ \sum_{k=1}^n \, (-1)^{n-k} \binom{n-1}{k-1} \,
    \frac{\partial B^+_{k-1}}{\partial x}(x) \, \\
    & & \\
    & & + \, \sum_{k=1}^n \, (-1)^{n-k} \binom{n-1}{k-1} \, B_{k-1}(x) \, - \, f^+_{n-1}
    (x) \\ & & \\
    &=& S_n(x) \, + \, 2x(1-x) \, \frac{\partial f^+_{n-1}}{\partial x}(x) \, + \,
    d^B_{n-1}(x) \, - \, f^+_{n-1}(x).
  \end{eqnarray*}
Finally, using equation (\ref{eq:localint+}), we compute that
\begin{eqnarray*}
  S_n(x) & = & 2x \sum_{k=1}^n \, (-1)^{n-k} \, k \, \binom{n-1}{k-1}  \,
  B^+_{k-1}(x) \, - \, 2x \sum_{k=1}^n \, (-1)^{n-k} \binom{n-1}{k-1} \,
  B^+_{k-1}(x) \\
  & & \\
  &=& 2x \sum_{k=1}^n \, (-1)^{n-k} \, k \,\binom{n}{k} \, B^+_{k-1}(x)
  \, - \, 2x \sum_{k=1}^{n-1} \, (-1)^{n-k} \, k \, \binom{n-1}{k} \, B^+_{k-1}(x) \\
  & & \\
  & & - \, 2x \, f^+_{n-1}(x) \\
  & & \\
  &=& 2nx \sum_{k=1}^n \, (-1)^{n-k} \binom{n-1}{k-1} \, B^+_{k-1}(x) \, - \,
  2(n-1)x \sum_{k=1}^n \, (-1)^{n-k} \binom{n-2}{k-1} \, B^+_{k-1}(x) \\
  & & \\
  & & - \, 2x \, f^+_{n-1}(x) \\
  & & \\
  &=& 2(n-1)x \, f^+_{n-1}(x) \, + \, 2(n-1)x \, f^+_{n-2}(x)
\end{eqnarray*}
and the proof follows.
\end{proof}

We will denote by $a^+_{n,k}$, $a^-_{n,k}$ and $d^B_{n,k}$ the coefficient of
$x^k$ in $f^+_n (x)$, $f^-_n (x)$ and $d^B_n (x)$, respectively. The following
recurrence relations can be derived from Proposition~\ref{prop:recf+n} and
\cite[Corollary~4.3]{CTZ09}.

\begin{corollary} \label{cor:ank}
For $n \ge 2$ and $k \ge 1$ we have
\begin{equation}\label{eq:a+nk}
a^+_{n,k} \ = \ (2k-1) a^+_{n-1,k} \, + \, 2(n-k) a^+_{n-1,k-1} \, + \, 2(n-1)
a^+_{n-2,k-1} \, + \, d^B_{n-1, k}
\end{equation}
and
\begin{equation}\label{eq:a-nk}
a^-_{n,k} \ = \ (2k-1) a^-_{n-1,k} \, + \, 2(n-k) a^-_{n-1,k-1} \, + \, 2(n-1)
a^-_{n-2,k-1} \, + \, d^B_{n-1, k-1}.
\end{equation}
\end{corollary}
\begin{proof}
Equation (\ref{eq:a+nk}) follows from the formula of
Proposition~\ref{prop:recf+n} by comparing the coefficients of
$x^k$. Since $a^-_{n, k} = d^B_{n, k} - a^+_{n, k}$, equation
(\ref{eq:a-nk}) follows from (\ref{eq:a+nk}) and the recurrence
relation for $d^B_{n, k}$ given in \cite[Corollary~4.3]{CTZ09}.
\end{proof}

\medskip
\noindent
\begin{proof}[Third proof of Theorem~\ref{thm:localint}] As in the second
proof, we denote by $\ell^+_n (x)$ and $\ell^-_n (x)$ the right-hand sides of
(\ref{eq:localint+}) and (\ref{eq:localint-}), respectively, and note that
$\ell^+_n (x) + \ell^-_n (x) = d^B_n (x)$ and that $\ell^+_n (x)$ is equal
to the local $h$-polynomial of $K_n$. In particular, we have $\ell^+_n (x)
= x^n \ell^+_n (1/x)$ by Theorem~\ref{thm:stalocal} (b). The proofs of
Proposition~\ref{prop:recf+n} and Corollary~\ref{cor:ank} show that the
coefficients of $\ell^+_n (x)$ and $\ell^-_n (x)$ satisfy (\ref{eq:a+nk}) and
(\ref{eq:a-nk}), respectively. Since $a^-_{n, k} = d^B_{n, k} - a^+_{n, k}$,
we may rewrite (\ref{eq:a+nk}) as
  $$ a^+_{n,k} \ = \ 2k a^+_{n-1,k} \, + \, 2(n-k) a^+_{n-1,k-1} \, + \, 2
     (n-1)a^+_{n-2,k-1} \, + \, a^-_{n-1, k}. $$
Switching $k$ to $n-k$ in this equality and using the symmetry $a^+_{n,k} =
a^+_{n,n-k}$ shows that
  $$ a^-_{n-1,k} \ = \ a^-_{n-1,n-k}. $$
Equivalently, we have $a^-_{n,k} = a^-_{n,n-k+1}$ for all $n$ and $k$ and
hence $\ell^-_n (x) = x^{n+1} \ell^-_n (1/x)$ for every $n \in \NN$. The
uniqueness of the defining properties of $f^+_n (x)$ and $f^-_n (x)$ shows
that $\ell^+_n (x) = f^+_n (x)$ and $\ell^+_n (x) = f^-_n (x)$ for every $n
\in \NN$.
\end{proof}

We have verified that $f^+_n (x)$ and $f^-_n (x)$ are real-rooted for $2 \le
n \le 10$. Thus, it is natural to conjecture the following statement.

\begin{conjecture} \label{conj:f+-n}
The polynomials $f^+_n (x)$ and $f^-_n (x)$ are real-rooted for every $n
\ge 2$.
\end{conjecture}

\section*{Acknowledgments}
The authors wish to thank Ron Adin, Benjamin Nill, Yuval Roichman, John
Stembridge, Mirk\'{o} Visontai and Volkmar Welker for useful pointers to the 
literature. The second author also thanks Francesco Brenti and Mirk\'{o} 
Visontai for useful discussions. The second author was co-financed by the 
European Union (European Social Fund - ESF) and Greek national funds through 
the Operational Program ``Education and Lifelong Learning" of the National 
Strategic Reference Framework (NSRF) - Research Funding Program: Heracleitus 
II. Investing in knowledge society through the European Social Fund.

\end{document}